\documentclass{article}

\usepackage[final]{ProbNum26} 

\usepackage{amsfonts}
\usepackage{amsmath} 
\usepackage{enumitem}  
\usepackage{amsthm}
\usepackage{setspace}
\usepackage{subcaption}

\usepackage{tikz}
\usetikzlibrary{positioning,arrows.meta}

\newtheorem{theo}{Theorem}[section]

\newtheorem{lemm}[theo]{Lemma}%
{\tiny }%

\theoremstyle{remark}
\newtheorem{remark}{Remark}

\usepackage{comment}

\def\X{{\mathcal{X}}}
\def\Y{{\mathcal{Y}}}

\def\N{{\mathcal{N}}}

\def\e{{\varepsilon}}

\usepackage[round]{natbib}

\usepackage[capitalise,nameinlink]{cleveref}

\usepackage{kantlipsum}

\probnumtitle{Bayesian Inference of Discretization Error Means in ODEs via Ensemble Kalman Filtering}
\probnumauthors{%
\name{Shoji Toyota}%
\affiliation{Department of Advanced Information Technology, Kyushu University, Japan}
\and%
\name{Yuto Miyatake}%
\affiliation{D3 Center, The University of Osaka, Japan}
}

\probnumabstract{We propose a Bayesian framework to quantify discretization errors in numerical solutions of ODE models based on observational data. The discretization error is modeled as a random variable, and its mean---referred to as the discretization error mean---is inferred from the observations. By introducing a Markov prior on the temporal evolution of the discretization error mean, we formulate the problem as a state-space model with a linear Gaussian observation process, which enables efficient inference via the Ensemble Kalman Filter. We also propose a specific form of a Markov prior motivated by classical discretization error analysis, in which global errors accumulate from local errors. 
The proposed prior depends on the step size of a numerical solver, and we establish its convergence rate in probability as the step size tends to zero. Numerical experiments on the pendulum system and the FitzHugh--Nagumo model demonstrate the effectiveness of the proposed approach.}

\begin{document}

\section{Introduction}
\label{sec:general_layout}

We consider the problem of finding a solution to a
$d_{\mathcal{X}}$-dimensional ordinary differential equation (ODE)
\begin{equation}\label{eq:ODE}
\frac{dx(t)}{dt} = f\bigl(x(t)\bigr),
\end{equation}
where the solution $x(t)$ takes values in $\mathbb{R}^{d_{\mathcal{X}}}$ and
$f : \mathbb{R}^{d_{\mathcal{X}}} \to \mathbb{R}^{d_{\mathcal{X}}}$ is a continuous map.
In the general case where $f$ is nonlinear, a closed-form solution is typically unavailable.
A classical approach to this problem is to discretize the ODE using numerical solvers,
such as the Euler method or Runge--Kutta methods.
Such discretization introduces numerical errors, often referred to as \emph{discretization errors}.
Understanding the behavior of these errors
is a central topic in numerical analysis
\citep{butcher2016numerical, hairer2020solving, iserles2009first}. 
The main focus is typically on analyzing its asymptotic behavior as the step size approaches zero,
or on deriving bounds on the discretization error. 

In contrast, methods for directly quantifying discretization error have become increasingly important. A typical example arises in the estimation of parameters in ODE models from observational data. Practical approaches, such as simulation-based inference \citep{cranmer2020frontier} or four-dimensional variational data assimilation (4D-Var)~\citep{asch2016data}, often rely on numerical solvers. In these methods, a parameter is chosen so that the resulting numerical solution fits the observations, implicitly assuming sufficient numerical accuracy. However, this assumption may be violated in certain settings, such as chaotic systems or large-scale problems
\citep{conrad2017statistical}. In such cases, discretization error must be explicitly quantified and properly incorporated into the estimation procedure.

Motivated by the need to quantify discretization error, recent research has approached this problem from statistical and machine learning perspectives. 
Prominent examples include Bayesian ODE 
solvers~\citep{beck2024diffusiontemper, kersting2020, le2025modelling, schmidt2021sir, SchoberSarkkaHennig2018, pmlr-v162-tronarp22a, Tronarp2018, TroSarHen2021, dingling2025propagating, kramer2025adaptive}, which formulate ODE discretization as Gaussian process inference. Another important direction is perturbation-based methods~\citep{abdulle2020random, conrad2017statistical, lie2022randomised, lie2019strong}, where stochastic perturbations are introduced into numerical solutions to represent uncertainty. These developments have led to an interdisciplinary field at the interface of numerical analysis and statistics, known as \emph{Probabilistic Numerics}.

Another line of research on quantifying discretization error in ODEs is known as the \emph{discretization error variance} approach~\citep{MARUMO2024modelling, Matsuda2019estimation, miyatake2025quantifying, toyota2025joint}. 
Within this framework, conventional numerical solvers
are used to obtain numerical solutions, while the resulting discretization errors are modeled as random variables. 
The associated variances, referred to as  \emph{discretization error variances}, are regarded as statistical parameters and estimated from observations. 
In contrast to Bayesian ODE solvers and perturbation-based methods, this approach directly infers the discretization error itself based on observations.

Motivated by the discretization error variance approach, we propose a framework for discretization error quantification, which we term the discretization error {\em mean} approach. 
In this setting, the discretization error is modeled as a random variable, and its {\em mean} is estimated from observational data. 
By focusing on the mean, the proposed method enables prediction of the direction of the discretization error, whereas approaches based on the variance permit inference only about its absolute magnitude. 
Inference is carried out within a Bayesian framework, where a prior is placed on the discretization error mean and the corresponding posterior is derived.

A key component of the proposed method is Bayesian inference of the discretization error mean using the Ensemble Kalman Filter (EnKF), a standard technique in data assimilation \citep{burgers1998analysis, evensen2009data}. By imposing a Markov prior on 
the discretization error means, the inference problem can be formulated as a state-space model whose observation process is linear Gaussian, to which the EnKF can be applied.

The specification of a Markov prior plays a central role in the effectiveness of the proposed framework. 
In previous work, \cite{toyota2025joint} proposed a Markov prior on discretization error variances based on insights from classical discretization error analysis for ODEs, namely that the global discretization error can be interpreted as the accumulation of local errors generated at each time step. 
We show that the prior introduced in \cite{toyota2025joint} can be naturally extended to the discretization error {\em mean}. 
The prior depends explicitly on the step size of the numerical solver; we analyze its asymptotic behavior as the step size tends to zero, following arguments analogous to those developed in \cite{toyota2025joint}.

The main contributions of this paper are as follows:
\begin{itemize}[topsep=0pt, partopsep=0pt, itemsep=5pt, parsep=0pt, leftmargin=5pt]

\item We introduce a statistical framework for quantifying discretization error in ODEs, termed the \emph{discretization error mean} approach, which enables inference on the discretization error from observational data, including its direction in addition to its magnitude.

\item We formulate the inference problem as a state-space model by imposing a Markov prior on the time evolution of the discretization error mean, thereby enabling Bayesian inference via the Ensemble Kalman Filter.

\item We propose a Markov prior for the discretization error mean by extending the approach of \cite{toyota2025joint}, and examine its asymptotic behavior as the step size of the numerical solver tends to zero.

\item Through numerical experiments, we demonstrate that the proposed method 
accurately quantifies both the magnitudes and directions of discretization errors.
\end{itemize}

\section{Proposed Method}
\subsection{Problem Setting}
We first describe the problem setting addressed in this paper. For a $d_{\mathcal{X}}$-dimensional ODE~\eqref{eq:ODE}, let its exact solution be denoted by $x(t)$. 
In this study, we numerically solve the equation at discrete time points $t_1, \ldots, t_N$, and denote the resulting numerical solutions by $x_{t_1}, \ldots, x_{t_N}$. 

The goal of this study is to quantify the discretization error $x(t_i) - x_{t_i}$. 
While its behavior has been a major 
subject of theoretical study
in numerical analysis~\citep{butcher2016numerical, hairer2020solving, iserles2009first}, 
typical analyses only derive bounds on 
the error or study its asymptotic properties as the step size approaches zero, without providing a direct quantification of the discretization error.

In this paper, we quantify the discretization error $x(t_i) - x_{t_i}$ 
using time-series observations $y^*_{t_1}, \ldots, y^*_{t_N}$ 
at discrete time points $t_1, \ldots, t_N$. 
These observations are assumed to be generated by the process
\begin{equation}\label{eq:observation}
    y_{t_i}^* = H x(t_i) + \varepsilon_{t_i},
\end{equation}
where $H$ is a full-rank linear observation matrix, and the observation noise 
$\varepsilon_{t_i} \sim \mathcal{N}(0, \Gamma)$ 
is Gaussian with zero mean and a positive definite covariance matrix $\Gamma$.

\subsection{Discretization Error Mean}
To accomplish our objective, we propose 
the \emph{discretization error mean} approach.
In this framework, we model the discretization error $x(t_i) - x_{t_i} $, the difference between the true solution $x(t_i)$ and the numerical solution $x_{t_i}$, as a Gaussian random variable:
\begin{equation}\label{eq:discrizatization_error}
    x(t_i) - x_{t_i} \sim \mathcal{N}(\mu_{t_i}, \Sigma),
\end{equation}
or  equivalently, 
\begin{equation}\label{eq:observation_modif}
     x(t_i) \sim \mathcal{N}(x_{t_i} + \mu_{t_i}, \Sigma),
\end{equation}
where $\Sigma$ is a fixed covariance matrix, and $\mu_{t_i}$ denotes the mean, which is estimated from observations and referred to as the {\em discretization error mean}. 
The discretization error mean is estimated using the likelihood 
$p(y_{t_1},...,y_{t_N}  \mid \mu_{t_0},...,\mu_{t_N} ) $.
Here, under the 
models in \eqref{eq:observation} and \eqref{eq:observation_modif}, the likelihood  can be expressed as 
\begin{align*}
p(y_{t_1},..., y_{t_N} |\mu_{t_0}, ..., \mu_{t_N} ) 
 = \prod_{i=1}^N p(y_{t_i} | \mu_{t_i}),
\end{align*}
where $p(y_{t_i} | \mu_{t_i})$ is defined by 
\begin{align}
&p(y_{t_i}\mid \mu_{t_i}) = \mathcal N \left(
y_{t_i};
\, H(x_{t_i}+\mu_{t_i}),
\, H \Sigma H^\top + \Gamma
\right), \label{eq:each_likelihood} 
\end{align}
a normal distribution with its mean and covariance given by $H(x_{t_i}+\mu_{t_i})$ and 
$H \Sigma  H^\top + \Gamma$ 
respectively. By utilizing this likelihood, we can estimate the discretization error mean $\mu_{t_i}$, enabling us to quantify the discretization error from observations $y_{t_1}^*, ..., y_{t_N}^*$. 

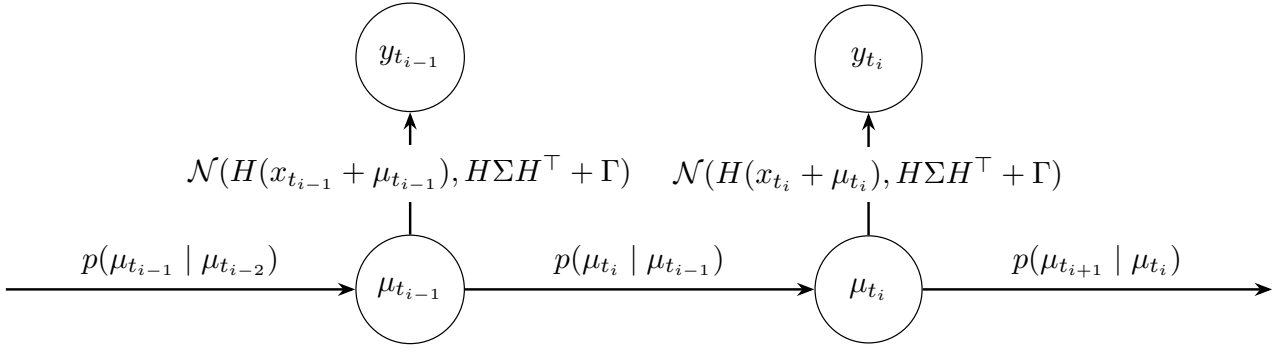
\begin{figure*}[t]
    \centering
\resizebox{\linewidth}{!}{
    \begin{tikzpicture}
    \node[circle, draw, text width=0.9cm,align=center] (c1) {$\mu_{t_{i-1}}$};
    \node[circle, draw, text width=0.9cm,align=center, right = 4.0cm of c1] (c2) {$\mu_{t_i}$};
    \node[circle, draw, text width=0.9cm,align=center, above = 1.4cm of c1] (c3) {$y_{t_{i-1}}$};
    \node[circle, draw, text width=0.9cm,align=center, above = 1.4cm of c2] (c4) {$y_{t_i}$};
    \node[left = 4.0cm of c1] (c5) {};
    \node[right = 4.0cm of c2] (c6) {};

    \draw [-{Stealth[length=2mm]}, thick] (c1)--node[above, fill=white,sloped]{$p(\mu_{t_i}\mid \mu_{t_{i-1}})$} (c2);
    \draw [-{Stealth[length=2mm]}, thick] (c1)--node[fill=white,sloped,rotate=-90]{$\mathcal{N}(H(x_{t_{i-1}} + \mu_{t_{i-1}}),H \Sigma H^\top + \Gamma)$} (c3);
    \draw [-{Stealth[length=2mm]}, thick] (c2)--node[fill=white,sloped,rotate=-90]{$\mathcal{N}(H (x_{t_{i}} + \mu_{t_{i}}),H \Sigma H^\top + \Gamma)$} (c4);
    \draw [-{Stealth[length=2mm]}, thick] (c5)--node[above, fill=white,sloped]{$p(\mu_{t_{i-1}}\mid \mu_{t_{i-2}})$}(c1);
    \draw [-{Stealth[length=2mm]}, thick] (c2)--node[above, fill=white,sloped]{$p(\mu_{t_{i+1}}\mid \mu_{t_i})$}(c6);
    \end{tikzpicture} }
    
    \caption{Reformulation of the objective in~(\ref{eq: targetproblem}) as a state-space model.
The latent transition follows the Markov prior
$p(\mu_{t_{i+1}} \mid \mu_{t_i})$, and the observation model is
$\mathcal{N}\!\big(H(x_{t_i} + \mu_{t_i}),\, H \Sigma H^\top  + \Gamma\big)$.
Equivalently, the observation equation can be written as
$y_t = H \mu_t + \varepsilon_t$ with
$\varepsilon_t \sim \mathcal{N}\!\big(H x_t,\, H \Sigma H^\top  + \Gamma\big)$,
which allows the use of the Ensemble Kalman Filter.
}
    \label{fig:state_space2}
\end{figure*}

\begin{remark}
Modeling the discretization error as a random variable and estimating its moments from observations was originally proposed by \citet{Matsuda2019estimation} and has been further developed in subsequent works \citep{miyatake2025quantifying, MARUMO2024modelling, toyota2025joint}. In this framework, the discretization error is modeled as
\begin{equation}\label{eq:discrizatization_error_variance}
    x(t_i) - x_{t_i} \sim \mathcal{N}(0, \Sigma),
\end{equation}
where the covariance matrix 
$\Sigma$, 
referred to as the \emph{discretization error variance}, is treated as a statistical quantity to be inferred from observations, while the mean is fixed at zero for all $t_i$.
This formulation allows one to impose a monotonicity constraint on $\Sigma_{t_i}$ by treating the discretization error variances as time-varying quantities, which may be difficult to incorporate in our approach based on modeling the mean of the discretization error.
On the other hand, our 
discretization error mean approach
enables us to infer the direction of the error and thereby adjust the numerical solution toward the exact solution. Another advantage will become apparent in the state-space formulation described in the following subsection. 
\end{remark}

\subsection{State-Space Model Interpretation of Discretization Error Mean Estimation}

We estimate the discretization error mean $\mu_{t_i}$ from observations $y_{t_1}^*, \ldots, y_{t_N}^*$ in a Bayesian manner. Specifically, we place a prior $p(\mu_{t_0}, \ldots, \mu_{t_N})$ on the time series of discretization error means and evaluate the  posterior:
\begin{align}
&p(\mu_{t_0}, \ldots, \mu_{t_N} \mid y_{t_1}^*, \ldots, y_{t_N}^*) \notag \\
&\propto
p(y_{t_1}^*, \ldots, y_{t_N}^* \mid \mu_{t_0}, \ldots, \mu_{t_N})
p(\mu_{t_0}, \ldots, \mu_{t_N}).
\label{eq: targetproblem}
\end{align}
In this study, we adopt a Markov prior for the time series of discretization error means:
\begin{equation}
\label{eq:markov_prior}
p(\mu_{t_0}, \ldots, \mu_{t_N})
=
p(\mu_{t_0})
\prod_{i=0}^{N-1}
p(\mu_{t_{i+1}} \mid \mu_{t_i}).
\end{equation}

Under the Markov assumption, the discretization 
error mean $\mu_{t_i}$ and the observations $y_{t_i}^*$ form a state-space model (Fig.~\ref{fig:state_space2}). 
The dynamics of the latent state are described by the temporal evolution of the discretization error mean, $\mu_{t_i} \rightarrow \mu_{t_{i+1}}$, 
with a stochastic transition governed by 
$
p(\mu_{t_{i+1}} \mid \mu_{t_i}).
$
The observation process is specified by the likelihood
$
p(y_{t_i} \mid \mu_{t_i}),
$
as defined in \eqref{eq:each_likelihood}.

Under the state-space formulation, evaluating the posterior \eqref{eq: targetproblem} reduces to latent state estimation for a given state-space model. 
Among the available approaches, we employ the Ensemble Kalman Filter (EnKF) \citep{evensen2009data}.  The EnKF can be applied when the observation model of the target state-space model,
\( p(y_{t_i} \mid \mu_{t_i}) \), can be expressed as a linear transformation of \( \mu_{t_i} \) with additive Gaussian noise. 
In our case, the observation process can be written as
\begin{equation}\label{observation_process_reduced}
y_{t_i} = H \mu_{t_i} + \varepsilon_{t_i},
\end{equation}
where
$
\varepsilon_{t_i} \sim \mathcal{N}\left(
H x_{t_i},\;
H \Sigma H^\top + \Gamma
\right), 
$\footnote{While EnKF typically assumes that the observation noise $\varepsilon_{t_i}$ is Gaussian with zero mean, 
it remains applicable even when the noise has a non-zero mean, 
without any essential modification, as shown in the following subsections.}
 enabling us to apply EnKF to our inference.

\begin{remark}\label{remark;linear_observation}
An advantage of the discretization error mean approach, compared with the variance approach, is that it can be reduced to inference in a state-space model with a {\em linear Gaussian} observation process \eqref{observation_process_reduced}. In the context of the discretization error variance approach, \cite{toyota2025joint} impose a Markov prior on the discretization error variance and perform Bayesian inference for this variance 
by reducing the problem to a state-space model.
However, since the associated observation process cannot be expressed in linear Gaussian form, inference is carried out using a particle filter. 
It is known that particle filters tend to perform poorly in high-dimensional systems compared with the EnKF, and the EnKF is often preferable 
in such settings. Therefore, the ability to directly leverage the EnKF constitutes an important advantage of modeling the mean of the discretization error rather than its variance.
\end{remark}

Hereafter, the tuples $\{ y_{t_1}, \ldots, y_{t_i} \}$ and 
$\{ \mu_{t_0}, \ldots, \mu_{t_i} \}$ are abbreviated as 
$y_{1:i}$ and $\mu_{0:i}$, respectively. In the EnKF, the distribution $p(\mu_{t_i} \mid y_{1:i}^*)$ is approximated by
\[
p(\mu_{t_i} \mid y_{1:i}^*) 
\approx \frac{1}{N_e} \sum_{n=1}^{N_e} \delta_{\hat{\mu}_{i \mid i}^{n}},
\]  
i.e., an ensemble of discrete points 
$\{ \hat{\mu}_{i \mid i}^{n} \}_{n=1}^{N_e}$. From the ensemble $\{ \hat{\mu}_{i \mid i}^{n} \}_{n=1}^{N_e}$ of $p(\mu_{t_i} \mid y_{1:i}^*)$, 
we construct the ensemble 
$\{ \hat{\mu}_{i+1 \mid i+1}^{n} \}_{n=1}^{N_e}$ at the next time $t_{i+1}$ 
through the two steps described below, referred to as the 
\emph{prediction step} (Subsection~\ref{subsec:prediction}) 
and the \emph{correction step} (Subsection~\ref{subsec:filtering}).  
Finally, a smoothing step is applied to obtain the 
posterior $p(\mu_{0:N} \mid y_{1:N}^*)$, incorporating all observations (Subsection~\ref{subsec:smoothing}).

\begin{algorithm*}
\begin{spacing}{1.2}
\begin{algorithmic}[1]

\State \textbf{Input:}
Prior $p(\mu_{t_0}) \prod_{i=0}^{N-1} p(\mu_{t_{i+1}}|\mu_{t_{i}})$,
observations $\{y_{t_i}^*\}_{i = 1}^N$,
ensemble size $N_e$,
lag length $L$. 

\State $~~~~~~~~~~$An ODE  $\frac{dx(t)}{dt} = f(x(t))$ and a numerical solver.

\State Generate an initial ensemble
$\{\hat{\mu}_{0|0}^n\}_{n=1}^{N_e} \sim p(\mu_{t_0})$

\For{$i = 0,\dots,N-1$}

    \For{$n=1,\dots,{N_e}$} \Comment{Prediction Step (Subsection \ref{subsec:prediction})}
        \State Simulate
        $\hat{\mu}_{i+1|i}^n
        \sim p(\mu_{t_{i+1}}\mid \mu_{t_{i}}=\hat{\mu}_{i|i}^n)$
    \EndFor

    \State Compute sample mean $\hat{\mu}_{i+1|i}$ and sample covariance 
    $\hat{\Sigma}_{i+1|i}$ from $\{ \hat{\mu}_{i+1|i}^n\}_{n = 1}^{N_e}$
    \For{$j=\text{max}(i+1 - L, 0),\dots,i+1$}    \\ \Comment{Smoothing is applied over $[0,\, i+1]$ if $i+1 \le L$, and over $[i+1-L,\, i+1]$ otherwise.}
        \State Compute cross-covariance
        $\hat{\Sigma}_{j, i+1 | i}$ from $\{ \hat{\mu}_{i+1|i}^n\}_{n=1}^{N_e}$ and $\{ \hat{\mu}_{j|i}^n\}_{n=1}^{N_e}$. 
        \State Compute Kalman gain
        $
        \hat{K}_{j|i}
        =
        \hat{\Sigma}_{j, i+1|i} H^\top
        \left(
        H \hat{\Sigma}_{i+1, i+1|i} H^\top + H \Sigma H^{\top}  + \Gamma
        \right)^{-1}
        $.
        \For{$n=1,\dots,N_e$}
        \State Draw
        $\hat{w}_{j}^n \sim \mathcal{N}(0, H \Sigma H^{\top}+\Gamma)$
            \State Update
            $
            \hat{\mu}_{j|i+1}^n  :=
\hat{\mu}_{j|i}^n
+
\hat{K}_{j|i}
\left(
y_{t_{i+1}}^*
+
\hat{w}^n_j
-
H \left( x_{t_{i+1}} + \hat{\mu}_{i+1|i}^n \right)
\right).
            $  \\
    
\Comment{Correction  Step (Subsection \ref{subsec:filtering}) with fixed lag smoothing (Subsection \ref{subsec:smoothing}).}
        \EndFor
    \EndFor
\EndFor
\State \Return $\{\hat{\mu}_{1|1+L}^n \}_{n=1}^{N_e}, \{ \hat{\mu}_{2|2+L}^n \}_{n=1}^{N_e},...,\{ \hat{\mu}_{N-L|N}^n \}_{n=1}^{N_e}, \{\hat{\mu}_{N-L+1|N}^n \}_{n=1}^{N_e},..., \{\hat{\mu}_{N|N}^n\}_{n=1}^{N_e}$,  \\
$\textcolor{white}{a}~~~~~~~~~~~~~~~~~~~~~~~~~~~$ ensembles $p(\mu_1|y_{1:1+L}^*), p(\mu_2|y_{1:2+L}^*),...,p(\mu_{N-L}|y_{1:N}^*), p(\mu_{N-L+1}|y_{1:N}^*),...,p(\mu_{N}|y_{1:N}^*)$.

\end{algorithmic}
\end{spacing}
\caption{Ensemble Kalman Filter for Discretization Error Quantification in an ODE Models}
\label{alg:enkf_smoothing}
\end{algorithm*}

\subsubsection{Prediction Step}\label{subsec:prediction}

Assume that we have an ensemble
$
\{ \hat{\mu}_{i \mid i}^n \}_{n=1}^{N_e}
$
that approximates the distribution $p(\mu_{t_i} \mid y_{1:i}^*)$.
In the prediction step, we construct a new ensemble
$
\{ \hat{\mu}_{i+1 \mid i}^n \}_{n=1}^{Ne}
$
to approximate
\[
p(\mu_{t_{i+1}} \mid y_{1:i}^*) = \int p(\mu_{t_{i+1}} | \mu_{t_{i}}) p(\mu_{t_{i}} \mid y_{1:i}^*) d\mu_{t_{i}}.
\]
This is achieved by propagating each ensemble member forward through the  transition distribution 
\( p(\mu_{i+1} \mid \mu_i) \).
Specifically, for $n = 1, \ldots, N_e$, we simulate
\[
\hat{\mu}_{i+1 \mid i}^n \sim
p\bigl( \mu_{i+1} \mid \mu_i = \hat{\mu}_{i \mid i}^n \bigr).
\]

\subsubsection{Correction Step}\label{subsec:filtering}

In this step, we construct an ensemble
$
\{ \hat{\mu}_{i+1 \mid i+1}^n \}_{n=1}^{N_e}
$
that approximates the posterior 
$
p(\mu_{t_{i+1}} \mid y_{1:i+1}^*),
$
given the ensemble
$
\{ \hat{\mu}_{i+1 \mid i}^n \}_{n=1}^{N_e}
$
obtained in the prediction step.

In this step, we first compute the covariance matrix 
\(
\hat{\Sigma}_{i+1 \mid i}
\)
from the ensemble
\(
\{ \hat{\mu}_{i+1 \mid i}^n \}_{n=1}^{N_e}
\)
as
\begin{align}
&\hat{\Sigma}_{i+1 \mid i} \label{eq:covariance_correction} \\
& :=
\frac{1}{N_e-1}
\sum_{n=1}^{N_e}
\left(
\hat{\mu}_{i+1 \mid i}^n
-
\hat{\mu}_{i+1 \mid i}
\right)
\left(
\hat{\mu}_{i+1 \mid i}^n
-
\hat{\mu}_{i+1 \mid i}
\right)^{\top}, \notag
\end{align}
where $\hat{\mu}_{i+1 \mid i}
:=
\frac{1}{N_e}
\sum_{n=1}^{N_e}
\hat{\mu}_{i+1 \mid i}^n$ denotes the sample mean.
With this covariance matrix $\hat{\Sigma}_{i+1 \mid i}$, we compute
\begin{align}
\hat{K}_{i+1 \mid i}
:=
\hat{\Sigma}_{i+1 \mid i} H^{\top}
\left(
H \hat{\Sigma}_{i+1 \mid i} H^{\top}
+ H \Sigma H^{\top}
+ \Gamma
\right)^{-1},
\label{eq:gain_correction}
\end{align}
which is referred to as the Kalman gain. Finally, we obtain an ensemble approximating
\(
p(\mu_{t_{i+1}} \mid y_{1:i+1}^*)
\)
by \emph{correcting} each ensemble member $\hat{\mu}_{i+1 \mid i}^n$
using the Kalman gain $\hat{K}_{i+1 \mid i}$ and the new observation
$y_{t_{i+1}}^*$:
\begin{align} &\hat{\mu}_{i+1 \mid i+1}^n := \label{eq:filter_ensemble} \\ &\hat{\mu}_{i+1 \mid i}^n + \hat{K}_{i+1|i} \left( y_{t_{i+1}}^* + \hat{w}^n_{i+1} - H\left(x_{t_{i+1}} + \hat{\mu}_{i+1 \mid i}^n \right) \right), \notag \end{align}
where
\(
\hat{w}_{i+1}^n
\)
are i.i.d.~samples from
\(
\mathcal{N}\bigl(0,\, H \Sigma H^{\top} + \Gamma \bigr).
\)

The following theorem ensures that the first and second moments of the updated
ensemble \eqref{eq:filter_ensemble}
coincide with those of $p(\mu_{t_{i+1}} \mid y_{1:i+1}^*)$
in the limit as $N_e \to \infty$, under the Gaussian approximation of
$p(\mu_{t_{i+1}} \mid y_{1:i+1}^*)$. A proof is provided in Appendix \ref{appendix:filtering}. 

\begin{theo}\label{eq:filtering_correspondance}
Assume that $p(\mu_{t_{i+1}} \mid y_{1:i}^*)$ is a Gaussian distribution 
$\mathcal{N}(\mu_{i+1|i}, \Sigma_{i+1|i})$ and that
$
\{ \hat{\mu}_{i+1 \mid i}^n \}_{n=1}^{N_e}
$
are i.i.d.\ samples from this distribution.
Then, the sample mean $\hat{\mu}_{i+1|i+1}$ and sample covariance $\hat{\Sigma}_{i+1|i+1}$
of the ensemble
$
\{ \hat{\mu}_{i+1 \mid i+1}^n\}_{n=1}^{N_e}
$ defined by \eqref{eq:filter_ensemble}
converge in probability to the true mean ${\mu}_{i+1|i+1}$
and covariance ${\Sigma}_{i+1|i+1}$ of
$p(\mu_{t_{i+1}} \mid y_{1:i+1}^*)$ as $N_e \to \infty$, 
i.e.,
\begin{align*}
    \hat{\mu}_{i+1|i+1} \xrightarrow{P} {\mu}_{i+1|i+1}, \quad
    \hat{\Sigma}_{i+1|i+1} \xrightarrow{P} {\Sigma}_{i+1|i+1}.
\end{align*}

\end{theo}

\subsubsection{Smoothing}\label{subsec:smoothing}

The prediction and correction steps yield 
$
p(\mu_{t_i} \mid y_{1:i}^*)
$ for $0 \leq i \leq N$. 
However, our goal is to evaluate the distribution
$
p(\mu_{t_i} \mid y_{1:N}^*),
$
the posterior conditioned on all observations at $t_1,..., t_N$.
While fixed-interval smoothing and other more advanced smoothing variants (e.g., \citet{kramer2025numerically}) are also available, we focus on fixed-lag smoothing, where
$
p(\mu_{t_i} \mid y_{1:i+L}^*)
$
is used as an approximation of $
p(\mu_{t_i} \mid y_{1:N}^*)$.

In fixed-lag smoothing, we introduce the augmented state
\begin{equation*}
{\boldsymbol{\mu}}_{t_{i}}
=
\left(
\mu_{t_{i}},\,
\mu_{{t_{i-1}}},\,
\ldots,\,
\mu_{{t_{i-L}}}
\right)^{\top}.
\end{equation*}
The dynamics of the augmented state from ${\boldsymbol{\mu}}_{t_i}$ to
\[
{\boldsymbol{\mu}}_{t_{i+1}}
=
\left(
\mu_{t_{i+1}},\,
\mu_{t_{i}},\,
\ldots,\,
\mu_{t_{i + 1 - L}}
\right)^{\top}
\]
are given as follows: the first component of ${\boldsymbol{\mu}}_{t_{i+1}}$ 
is assumed
to follow
$
\mu_{t_{i+1}}
\sim
p \left(
\mu_{t_{i+1}}
\mid
\mu_{t_{i}}
\right), 
$ and the remaining components are deterministically shifted from
\[
\left(
\mu_{t_{i}},\,
\mu_{{t_{i-1}}},\,
\ldots,\,
\mu_{{t_{i+1 - L}}}
\right)^{\top},
\]
which correspond to the first through $L$-th components of
${\boldsymbol{\mu}}_{t_i}$. By defining the observation model as
\begin{equation*}
p(y_{t_i} \mid {\boldsymbol{\mu}}_{t_i})
=
p(y_{t_i} \mid \mu_{t_i}),
\end{equation*}
we obtain an augmented state-space model for ${\boldsymbol{\mu}}_{t_i}$ and $y_{t_i}$. 
By solving this augmented state-space model with the EnKF, we obtain an ensemble approximation of 
$p({\boldsymbol{\mu}}_{t_i} \mid y_{1:i}^*)$. 
By extracting the final component of each ensemble member, we obtain an estimate of 
$p(\mu_{t_{i-L}} \mid y_{1:i}^*)$, 
where observations up to $L$ time steps ahead are incorporated into the estimation of $\mu_{t_{i-L}}$.

{For an ensemble of
$p({\boldsymbol{\mu}}_{t_{i+1}} \mid y_{1:i}^*)$ denoted by 
\begin{equation*}
  \left\{ \hat{\boldsymbol{\mu}}_{{i+1|i}}^n \right\} = \left\{\left( \hat{\mu}_{i+1|i }^n, \ldots, \hat{\mu}_{i+1-L|i}^n  \right)^{\top} \right\}_{n=1}^{N_e},
\end{equation*}
each ensemble member is updated via the fixed-lag smoothing step as
\begin{align}
\begin{split}
&\hat{\mu}_{j|i+1}^n  \\ & \ :=
\hat{\mu}_{j|i}^n
+
\hat{K}_{j|i}
\left(
y_{t_{i+1}}^*
+
\hat{w}^n_{j|i}
-
H \left( x_{t_{i+1}} + \hat{\mu}_{i+1|i}^n \right)
\right). 
\end{split}\label{eq:simplified_smoothing}
\end{align}
The derivation is given in Appendix~\ref{appendix:smoothing}. Here, $\hat{\Sigma}_{j, i+1 \mid i}$ denotes the sample cross-covariance between
$\{ \hat{\mu}_{j \mid i}^n \}_{n=1}^{N_e}$ and
$\{ \hat{\mu}_{i+1 \mid i}^n \}_{n=1}^{N_e}$ for $j = i+1 -L,..., i+1$, 
and  the smoothing Kalman gain $\hat{K}_{j|i}$ is given by
\begin{equation*}
\hat{K}_{j|i}
=
\hat{\Sigma}_{j, i+1|i} H^\top
\left(
H  \hat{\Sigma}_{i+1, i + 1\mid i}  H^{\top} + H \Sigma H^\top + \Gamma
\right)^{-1}.
\end{equation*}
}
\vspace{-5mm}

\section{A Markov Prior on Discretization Error Mean}
At this stage, we have not specified a Markov prior on the discretization error \emph{mean}. 
In \cite{toyota2025joint}, a Markov prior was introduced for the discretization error \emph{variances}, derived from an error propagation analysis in numerical analysis. 
Here, we adopt the same underlying idea to model the discretization error \emph{mean}. The exposition below follows that of \cite{toyota2025joint}, with the notation adapted to the current setting for the discretization error mean.

\subsection{Local Propagation of the Global Error}

In this subsection, we briefly review the standard theory of numerical error analysis.
For completeness and to fix notation, we summarize the distinction between the \emph{local error} and the \emph{global error}, and describe the propagation mechanism of the latter. 

We make the following assumptions:
\vspace{-1mm}
\begin{itemize}[topsep=5pt, partopsep=0pt, itemsep=5pt, parsep=0pt, leftmargin=11pt]
\item The observation time interval $t_{i+1} - t_i$ is constant.
\item The step size used in the time integrator is denoted by $h$, and $t_{i+1} - t_i = kh$ for some positive integer $k$.
\end{itemize}

The first assumption is introduced solely for clarity of exposition. 
If the second assumption does not hold in practice, the solution at $t = t_i$ can be approximated, for example, by interpolation techniques using neighboring numerical solutions~\citep{hairer2020solving}.

The numerical solution at $t = t_i + jh$ is denoted by $x_{i,j}$ so that $x_{i+1} = x_{i,k}$. 
The time-$h$ flow of the ODE is denoted by $\phi_h$, and the time-$h$ flow of the numerical solver by $\psi_h$. 
The local error is defined as the numerical error induced in a single time step, i.e.,
$\phi_h(x) - \psi_h(x)$. 
We define
\begin{equation*}
L(t_{i,j}) :=   \phi_h(x_{i,j}) - \psi_h(x_{i,j})
=   \phi_h(x_{i,j}) - x_{i,j+1}  .
\end{equation*}
On the other hand, the global error is defined by
\begin{equation*}
G(t_{i,j}) =  x(t_i + jh) - x_{i,j}.
\end{equation*}
We now examine how the global error propagates. Observe that 
\begin{align*} 
G(t_{i,j+1}) & = x(t_i+(j+1)h) - x_{i,j+1} 
\\ &= \phi_h(x(t_i+jh)) - \phi_h(x_{i,j}) + \phi_h(x_{i,j})  - x_{i,j+1} 
\\ & = \phi_h(x(t_i+jh)) - \phi_h(x_{i,j})  + L(t_{i,j}).  
\end{align*}
{\color{black}
The mean-value theorem implies that 
\begin{align*}
    & \phi_h(x(t_i+jh)) - \phi_h(x_{i,j}) \\
    &\quad = \bigg(\int_0^1 \nabla_x \phi_h (x_{i,j} + \theta G(t_{i,j})) \,d\theta  \bigg) G(t_{i,j}) \\
    & \quad \approx \nabla_x \phi_h (x_{i,j}) G(t_{i,j}).
\end{align*}
Here, 
as $\phi_h(x_{i,j}) = x_{i,j} + h f(x_{i,j}) + \mathrm{o} (h)$ by the definition of ODE \eqref{eq:ODE}, we see that
\begin{equation*}
    \nabla_x\phi_h(x_{i,j}) = I + h \nabla_x f(x_{i,j}) + \mathrm{o} (h).
\end{equation*}
By ignoring the term $\mathrm{o}(h)$, the propagation of the global error can be approximated by
\begin{align} 
&G(t_{i,j+1}) \approx (I +  h \nabla_x f(x_{i,j}) ) G(t_{i,j}) + L(t_{i,j}). \label{eq:ge_propagation1} 
\end{align}
}
\subsection{A Markov Prior}\label{subsec:markov_prior} 
In this section, we propose a Markov prior on discretization error means.
Note that the discretization error mean $\mu_{t_{i,j}}$ in our context can be interpreted as a model for the global error $G(t_{i,j})$. 
Motivated by \eqref{eq:ge_propagation1}, we arrive at the following Markov prior, in which the discretization error mean $\mu_{t_{i,j}}$ evolves according to the probabilistic transition below as time advances by $h$: 
\begin{equation}\label{eq:proposed_prior} 
\mu_{t_{i,j+1}} = M_{i, j}  \mu_{t_{i,j}} + \tilde{L}(t_{i,j}), 
\end{equation} 
for $M_{i, j} \sim P~(j = 0, \dots, k-1)$. Here, $P$ denotes a distribution over matrices $M_{i, j} \in \mathbb{R}^{d_\mathcal{X} \times d_\mathcal{X}}$.  The value of $\tilde{L}(t_{i,j})$ is an 
approximation
of $L(t_{i,j})$. 
This can be 
estimated by, for example, either $\tilde{L}(t_{i,j}) = \psi_{h/2} \circ \psi_{h/2} (x_{i,j}) - \psi_h (x_{i,j}) $ or $\tilde{L}(t_{i,j}) = \tilde{\psi}_h (x_{i,j}) - \psi_h (x_{i,j}) $, where $\tilde{\psi}_h$ denotes a higher order numerical solver. These techniques are commonly used to control the step size during time integration~\citep{hairer2020solving}.

The proposed Markov prior involves a distribution $M_{i,j} \sim P_\lambda$, which typically depends on a hyperparameter $\lambda$. In our experiment, for instance, we set $M_{i, j}= m_{i, j} \ I$, with $m_{i, j}$ drawn from a normal distribution $m_{i, j} \sim \N (\alpha, \beta^2)$; in this case, the hyperparameter $\lambda = (\alpha, \beta)$ must be properly chosen. In this study, we  
select the hyperparameter vector by maximizing an ensemble-based estimate of the marginal likelihood $p(y^*_{1:N} | \alpha, \beta)$.

As shown in the next section, the empirically optimal parameters often satisfy the condition 
$\mathbb{E}_{M \sim P_\lambda}[M] \approx I$,
as observed in the discretization error variance case~\citep{toyota2025joint}. 
This suggests that it suffices to restrict attention to hyperparameter candidates $\lambda$ satisfying
$\mathbb{E}_{M \sim P_\lambda}[M] = I$.
In the next subsection, we provide theoretical support for this empirical finding.

\subsection{Asymptotics of the Prior}

We study the asymptotics of the  prior~\eqref{eq:proposed_prior} as the step size $h$ tends to zero. 
For the sake of theoretical simplicity, we identify $\tilde{L}(t_{i,j})$ with the exact local error $L(t_{i,j})$, 
although 
in practice 
it is approximated using two numerical solvers. 
To make the dependence on the step size explicit, we denote the prior by
\begin{align}
p_h \left( \mu_{t_{0,0}}, \ldots, \mu_{t_{N,0}} \right)
&= p_h \left( \mu_{t_0}, \ldots, \mu_{t_N} \right)  \label{eq:prior_stepsize} \\
&= p_h \left( \mu_{t_0} \right) \prod_{i=0}^{N-1} 
   p_h \left( \mu_{t_{i+1}} \mid \mu_{t_i} \right). \notag
\end{align}
Numerical integrators are generally constructed so that their discretization error 
vanishes as $h \to 0$. 
In accordance with this property, the prior $p_h(\mu_{t_i})$ should be specified 
so that $\mu_{t_i}$ converges to $0$ in probability as $h \to 0$. 
The following theorem establishes this 
property and characterizes 
the associated convergence rate. 
The proof is given in Appendix~\ref{appendix_proof_prior}; the argument
closely parallels Theorem 4.1 of~\cite{toyota2025joint}.

\begin{theo}[Modification of {\normalfont \citet[Theorem~4.1]{toyota2025joint}}]
\label{theo_convergence_rate_revised}
Assume that there exist constants $C_M, C_L, C_\mu \in \mathbb{R}_{>0}$ satisfying  (I) $\mathbb{E}_{M\sim P_\lambda}\!\left[\|I - M\|_{\text{op}}\right]
\leq  C_M h,$ (II) $L(t_{i,j}) \leq C_L  h^{a+1}$, (III) $\mathbb{E}\!\left[\|\mu_{t_0}\| \right]
\leq C_\mu h^{b}.$
 Assume also that, for every $i$ and $j$, 
(IV) $M_{i, j} \sim P$ is independent of $\mu_{t_{i, j}}$.
Then, for all $i \ge 1$,
\[
\mu_{t_i}
=
\mathcal{O}_p\!\left(h^{\min({a},{b})}\right)
\qquad (h \to 0).
\]
\end{theo}

Condition~(I) ensures that the random matrix $M_{i,j}$
is a small perturbation of the identity matrix.
This is consistent with the empirical finding $\mathbb{E}_{M \sim P_\lambda}[M] \approx I$ shown in Tables \ref{table:hypara_pendulum} and \ref{table:hypara_FN},
especially for small step sizes $h$.
Condition~(II)  defines a given numerical method to be of order $a$,
i.e., the local error is $\mathcal{O}(h^{a+1})$.
Condition~(III) corresponds to the choice of a prior distribution $p_h(\mu_0)$ at the initial time $t_0$. 
This condition can be readily satisfied, since we can specify $p_h(\mu_0)$ flexibly.

This theorem shows that the proposed prior has a natural convergence rate that reflects insights from numerical analysis.
For many numerical methods, if the local error
satisfies $L(t_{i,j}) = \mathcal{O}(h^{a+1})$, then the global error
$G(t_{i,j})$ is $\mathcal{O}(h^{a})$, 
provided that the vector field of the underlying ODE is Lipschitz continuous
~\citep{butcher2016numerical, hairer2020solving, iserles2009first}.
Since the mean of the discretization error in this study corresponds to the global error, 
the proposed prior is expected to be $\mathcal{O}_p(h^{a})$ under Condition~(II). 
This theorem shows that this objective is achieved by constructing $P$ and 
$p(\mu_{t_0})$ to satisfy Conditions~(I) and~(III), respectively.

\begin{figure*}[t]
    \centering

    \begin{subfigure}{\textwidth}
        \centering
        \begin{minipage}{0.5\textwidth}
            \centering
            \includegraphics[width=\linewidth]{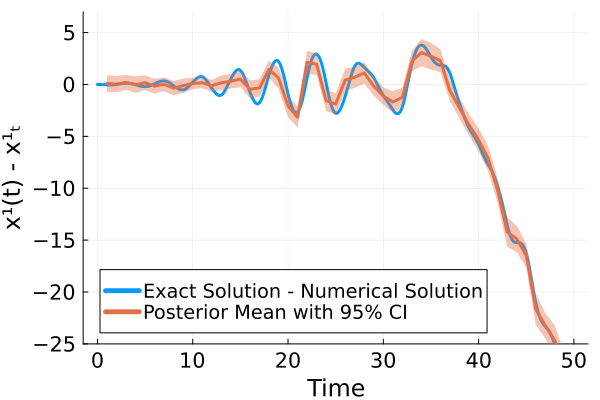}
        \end{minipage}\hfill
        \begin{minipage}{0.5\textwidth}
            \centering
            \includegraphics[width=\linewidth]{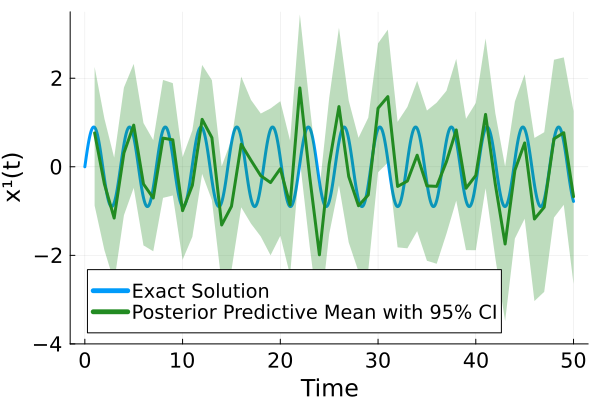}
        \end{minipage}
            \end{subfigure}
\vspace{-3mm}
\caption{Discretization error quantification results for $x^1(t)$ in the pendulum system.
\textbf{Left:} The mean and $95\%$ credible interval of the posterior $p(\mu_{t_i}^1 \mid y_{1:40}^*)$.
\textbf{Right:} The mean and $95\%$ credible interval of the posterior predictive distribution $p(x^1(t) \mid y_{1:40}^*)$.}
\label{fig:pendulum_quantification_result_dim1}

    \vspace{1em}

    \begin{subfigure}{\textwidth}
        \centering
        \begin{minipage}{0.5\textwidth}
            \centering
            \includegraphics[width=\linewidth]{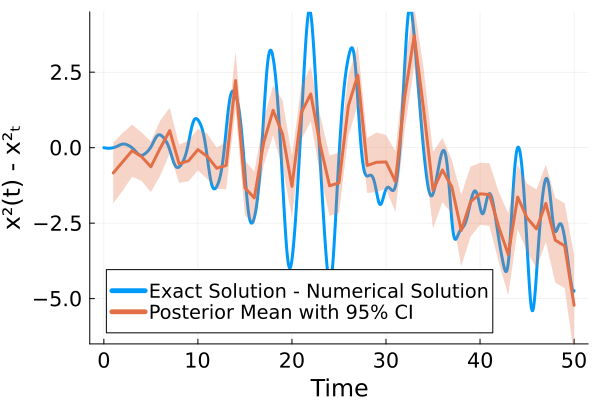}
        \end{minipage}\hfill
        \begin{minipage}{0.5\textwidth}
            \centering
            \includegraphics[width=\linewidth]{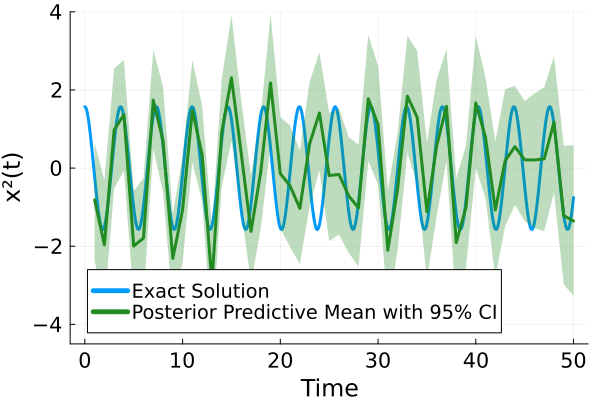}
        \end{minipage}
         \end{subfigure}
         \vspace{-3mm}
       \caption{Discretization error quantification results for $x^2(t)$ in the pendulum system.
\textbf{Left:} The mean and $95\%$ credible interval of the posterior $p(\mu_{t_i}^2 \mid y_{1:40}^*)$.
\textbf{Right:} The mean and $95\%$ credible interval of the posterior predictive distribution $p(x^2(t) \mid y_{1:40}^*)$.}
\label{fig:pendulum_quantification_result_dim2}
\vspace{-3mm}
\end{figure*}

\begin{remark}
The same set of conditions and the corresponding convergence rate were established in 
\cite{toyota2025joint} (Theorem~4.1) for the discretization error \emph{variance}. 
Theorem~\ref{theo_convergence_rate_revised} demonstrates that the analogous asymptotic property 
derived in~\cite{toyota2025joint} extends to the discretization error \emph{mean} setting 
considered in the present study. 
We also note that the nonnegativity assumption on $M$ imposed in~\cite{toyota2025joint} 
is relaxed here, thereby allowing for a broader class of prior constructions.
\end{remark}

\section{Experiment}

We evaluate the proposed method on the pendulum system and the FitzHugh--Nagumo model, with an additional Lorenz-96 experiment presented in Appendix~\ref{appendix_Lorenz96}.
The ODEs are solved using the explicit Euler method, which is of order one, and the associated discretization errors are evaluated. 
For local error estimation, we employ the Runge method, which is 
of order two\footnote{Depending on the baseline solver, using a solver of the same or lower order may be sufficient. The key requirement is to compare numerical solutions from two different solvers.}. 
A distribution $M_{i, j} \sim P$ in the proposed prior is assumed to be  $m_{i, j}  I$, where 
$m_{i, j} $ 
follows a 
univariate normal distribution $\mathcal{N}(\alpha, \beta^2)$.  The covariance matrix $\Sigma$ representing the discretization error \eqref{eq:discrizatization_error} is assumed to be $\gamma  I$, where $\gamma \in \mathbb{R}$. 
These parameters $(\alpha, \beta, \gamma) \in \mathbb{R}^3$ are selected by maximizing the marginal log-likelihood
$
\log p(y^*_{1:N} \mid \alpha, \beta, \gamma).
$
The covariance matrix of the observation process \eqref{eq:observation} is fixed to
$
\Gamma = \mathrm{diag}(1.0^2, 1.0^2),
$
and the ensemble size $N_e$ and the smoothing lag $L$ are $100$ and $10$, respectively. We assess the proposed method with a single observation trajectory, comparing the inferred discretization errors with the exact ones (more precisely, the error between a highly accurate reference solution and the numerical solution).

\subsection{Pendulum System}
We consider the second order ODE
    $\frac{d x^2(t)}{dt^2} = - \frac{g}{l} \sin x (t).$
In this section, we solve the equivalent first order system
\begin{equation}
\frac{d}{dt}
\left[
\begin{array}{c}
x^1 (t) \\
x^2 (t)
\end{array}
\right]
= \left[
\begin{array}{c}
x^2 (t) \\
- \frac{g}{l}  \sin (x^1 (t))
\end{array}
\right],
\end{equation}
where $l$ is fixed to $3.0$. 
We use the Euler method with a step size of $0.05$ to numerically solve the system. 
The observation operator $H$ is defined as
\[
H =
\begin{bmatrix}
1.0 & 2.0 \\
2.0 & 1.0
\end{bmatrix}.
\]
This operator is not one that is typically used in practice; rather, it was introduced to examine the performance of the method beyond the diagonal case.
We assume that observations are available at $t = 1, 2, \ldots, 40$, with a fixed covariance matrix 
$\Gamma = \operatorname{diag}(1.0^2,\, 1.0^2)$. We also set $p(\mu_{t_0})$ to be a 2-dimensional standard normal distribution,
$
p(\mu_{t_0}) = \mathcal{N}(0, I).
$

Table~\ref{table:hypara_pendulum} lists the top $10$ parameters
$(\alpha, \beta, \gamma)$ with corresponding log-likelihoods,
among the candidates $\alpha \in \{-1.4, -1.2, \dots, 1.4\}$, $\beta \in \{0.05, 0.10, \dots, 0.5\}$ and $\gamma \in \{0.5, 1.0, \dots, 3.0\}$.
We observe that 
high-likelihood parameter combinations tend to have $\alpha = 1$, 
with a few cases at $\alpha = -1$. This is consistent with the theoretical implication of
Theorem~\ref{theo_convergence_rate_revised}. The standard-deviation parameter $\beta$ ranges from $0.25$ to $0.40$, while $\gamma$ takes $0.50$ and $1.00$. In the following numerical experiments, we use the choice $(\alpha, \beta, \gamma) = (1.0, 0.3,0.5)$, which achieved the highest log-likelihood among all candidates.

\begin{table}[t]
  \centering
  \setlength{\tabcolsep}{7pt}
  \caption{Top 10 hyperparameters $(\alpha, \beta, \gamma)$ with corresponding 
  log-likelihoods $\log p(y^*_{1:40} \mid \alpha, \beta, \gamma)$ in the pendulum system.}\label{table:hypara_pendulum}
  \begin{tabular}{ccc}
    \toprule
    Rank & $(\alpha,\beta,\gamma)$ & Log-likelihood \\
    \midrule
     1  & $(1.000,\ 0.300,\ 0.500)$ & $-231.731995$ \\
     2  & $(1.000,\ 0.350,\ 0.500)$ & $-233.595016$ \\
     3  & $(1.000,\ 0.250,\ 0.500)$ & $-237.323069$ \\
     4  & $(1.000,\ 0.400,\ 0.500)$ & $-241.765176$ \\
     5  & $(1.000,\ 0.350,\ 1.000)$ & $-241.830257$ \\
     6  & $(1.000,\ 0.300,\ 1.000)$ & $-244.841849$ \\
     7  & $(-1.000,\ 0.400,\ 0.500)$ & $-245.983714$ \\
     8  & $(1.000,\ 0.250,\ 1.500)$ & $-246.298006$ \\
     9  & $(1.000,\ 0.400,\ 1.000)$ & $-247.887706$ \\
    10  & $(-1.000,\ 0.350,\ 0.500)$ & $-248.311574$ \\
    \bottomrule
  \end{tabular}
  \vspace{-3mm}
\end{table}

The left panels in Figures~\ref{fig:pendulum_quantification_result_dim1} 
and \ref{fig:pendulum_quantification_result_dim2} show the estimated posterior 
on $\mu_{t_i}$ with the true discretization error. 
We confirm that the proposed method successfully captures the true discretization error. 
The right panels in Figures~\ref{fig:pendulum_quantification_result_dim1} 
and \ref{fig:pendulum_quantification_result_dim2} show the posterior predictive mean and the $95\%$ credible intervals (CIs), computed from $100$ samples drawn from the posterior predictive distribution $p(x(t) \mid y_{1:40}^*)$.
Specifically, we randomly draw samples $\hat{\mu}_{1:40}$ from the  posterior 
$p(\mu_{1:40} \mid y_{1:40}^*)$, and then generate samples of the discretization error as
$
r_{t_i} \sim \mathcal{N}(\hat{\mu}_{t_i}, \Sigma),
$
following the definition of the discretization error mean in \eqref{eq:discrizatization_error}. 
Unlike in the left panels, we add the numerical solution $x_{t_i}$ to $r_{t_i}$ 
and compare $x_{t_i} + r_{t_i}$ with the exact solution  $x(t_i)$.
We observe that, while the estimated solution occasionally deviates from the exact solution, particularly in the second dimension around $t = 20\text{--}30$ and  $t = 40\text{--}50$,  the resulting trajectory follows the exact solution, indicating that the proposed method effectively reconstructs it.

\begin{figure*}[t]
    \centering

    \begin{subfigure}{\textwidth}
        \centering
        \begin{minipage}{0.5\textwidth}
            \centering
            \includegraphics[width=\linewidth]{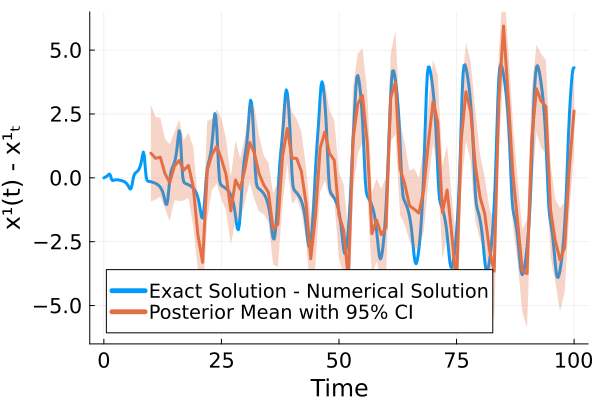}
        \end{minipage}\hfill
        \begin{minipage}{0.5\textwidth}
            \centering
            \includegraphics[width=\linewidth]{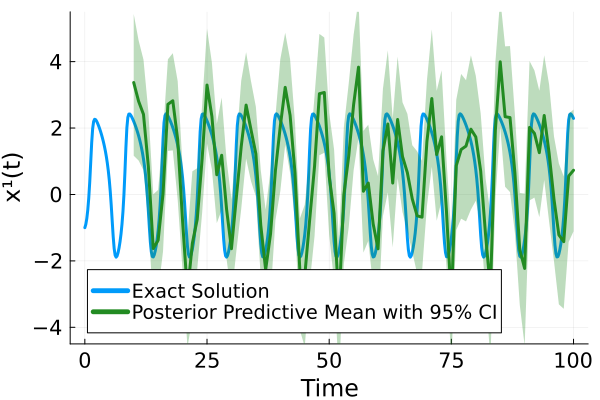}
        \end{minipage}
    \end{subfigure}
    \vspace{-4mm}
\caption{Discretization error quantification results for $x^1(t)$ in the FN model.
\textbf{Left:} The mean and $95\%$ credible interval of the posterior distribution $p(\mu_{t_i}^1 \mid y_{10:50}^*)$.
\textbf{Right:} The mean and $95\%$ credible interval of the posterior predictive distribution $p(x^1(t) \mid y_{10:50}^*)$.}\label{fig:FN_quantification_result_dim1}

    \begin{subfigure}{\textwidth}
        \centering
        \begin{minipage}{0.5\textwidth}
            \centering
            \includegraphics[width=\linewidth]{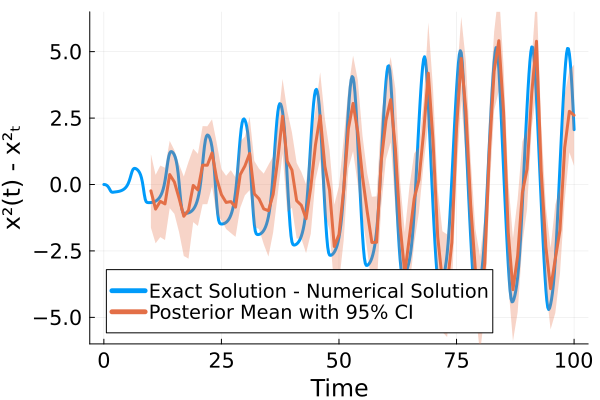}
        \end{minipage}\hfill
        \begin{minipage}{0.5\textwidth}
            \centering
            \includegraphics[width=\linewidth]{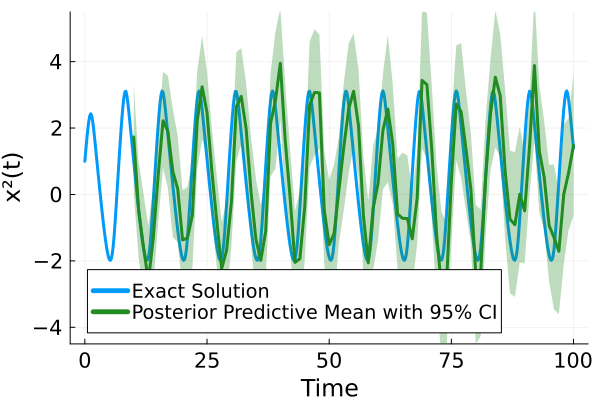}
        \end{minipage}
    \end{subfigure}
    \vspace{-3mm}
\caption{Discretization error quantification results for $x^2(t)$ in the FN model.
\textbf{Left:} The mean and $95\%$ credible interval of the posterior distribution $p(\mu_{t_i}^2 \mid y_{10:50}^*)$.
\textbf{Right:} The mean and $95\%$ credible interval of the posterior predictive distribution $p(x^2(t) \mid y_{10:50}^*)$.}
\label{fig:FN_quantification_result_dim2}
\vspace{-6mm}
\end{figure*}

\subsection{FitzHugh-Nagumo model}
We consider the FitzHugh-Nagumo model given by 
\begin{equation*}
\frac{d}{dt}
\left[
\begin{array}{l}
x^1(t) \\
x^2(t)
\end{array}
\right]
=
\left[
\begin{array}{l}
c  \biggl( \rule{0pt}{2.5ex} x^1(t) - \dfrac{\left(x^1(t) \right)^3}{3} + x^2(t) \biggr) \\[2.ex]
\rule{0pt}{2.5ex} -\dfrac{1}{c} \biggl( x^1\left(t \right) - a + b x^2\left(t \right) \biggr)
\end{array}
\right]
\end{equation*}
with $(a, b, c)=(0.5, -0.2, 1.0)$.
We solve the equation using the Euler method with a step size of $0.2$. 
The observation operator is set to $H = \mathrm{diag}(3.0, 3.0)$, and observations are available at $t = 10, \ldots, 50$. 
We set $p(\mu_{t_0})$ to be a 
two dimensional
normal distribution,
$
p(\mu_{t_0}) = \mathcal{N}(0, I).
$ 

Table~\ref{table:hypara_FN} lists the top $10$ hyperparameters $(\alpha, \beta, \gamma)$ and their corresponding log-likelihoods among the  candidates
$\alpha \in \{-1.4, -1.2, \dots, 1.4\}$, $\beta \in \{0.1, 0.2, \dots, 1.0\}$, and $\gamma \in \{0.5, 1.0, \dots, 4.0\}. $
We observe that all selected values of $\alpha$ are $1.0$, consistent with the tendency observed in the pendulum system.
The standard-deviation parameter $\beta$ takes values $0.3$ and $0.4$, while no clear tendency is observed for $\gamma$.
Henceforth, we adopt $(\alpha, \beta, \gamma) = (1.0, 0.3, 0.4)$, which achieved the highest log-likelihood among all candidates.

The left panels of Figs.~\ref{fig:FN_quantification_result_dim1} and \ref{fig:FN_quantification_result_dim2} show the estimated posterior distributions in comparison with the true discretization errors. 
We confirm that the proposed method generally captures the trajectory of the true discretization errors. 
The right panels of the same figures present the posterior predictive distributions of the exact solutions. 
The resulting trajectories match the exact solutions, indicating that the proposed method effectively corrects the numerical solutions toward the exact ones.

\begin{table}[t]
  \centering
  \setlength{\tabcolsep}{7pt}
  \caption{Top 10 parameter triplets $(\alpha, \beta, \gamma)$ with corresponding 
  log-likelihoods $\log p(y^*_{10:50} \mid \alpha, \beta, \gamma)$ 
  for the FN model.}\label{table:hypara_FN}
  \begin{tabular}{ccc}
    \toprule
    Rank & $(\alpha,\beta,\gamma)$ & Log-likelihood \\
    \midrule
     1  & $(1.000,\ 0.300,\ 4.000)$ & $-556.734585$ \\
     2  & $(1.000,\ 0.300,\ 2.500)$ & $-557.755289$ \\
     3  & $(1.000,\ 0.300,\ 1.500)$ & $-558.155779$ \\
     4  & $(1.000,\ 0.300,\ 3.500)$ & $-558.311221$ \\
     5  & $(1.000,\ 0.300,\ 2.000)$ & $-558.649372$ \\
     6  & $(1.000,\ 0.300,\ 1.000)$ & $-559.566540$ \\
     7  & $(1.000,\ 0.300,\ 3.000)$ & $-561.452090$ \\
     8  & $(1.000,\ 0.300,\ 0.500)$ & $-564.933778$ \\
     9  & $(1.000,\ 0.400,\ 4.000)$ & $-568.680713$ \\
    10  & $(1.000,\ 0.400,\ 3.500)$ & $-571.341663$ \\
    \bottomrule
  \end{tabular}
  \vspace{-2mm}
\end{table}

\section{Conclusion}
We propose a Bayesian framework to quantify discretization errors in ODEs through the \emph{discretization error mean}. 
Our method enables inference of both the magnitudes and directions of the errors, allowing the numerical solution to be adjusted toward the exact solution.
By introducing a Markov prior on the  error mean motivated by classical error analysis, we perform inference using the Ensemble Kalman Filter. 
Numerical experiments demonstrate that the proposed method effectively captures discretization errors. Our approach relies on the assumption that the underlying model is consistent with the observations; handling model misspecification is an important  further work. 
Jointly modeling the mean and variance of discretization errors would be another interesting  extension, allowing for a more flexible representation of numerical uncertainty.

\section*{Acknowledgements}
We thank Dr.~Shin'ya Nakano and Dr.~Keisuke Yano at the Institute of Statistical Mathematics for valuable discussions. In particular, we thank Dr.~Shin'ya Nakano for insightful comments regarding the marginal likelihood computation for the EnKF.
This work is supported by JSPS KAKENHI Grant Numbers 24K02951, 24K00540, 25H00449, 24K20750, 25K21806, JP25H01454 and JST ACT-X, Japan, Grant Number JPMJAX25CH. Miyatake is also supported by JST (Moonshot R\&D Program) Japan Grant Number JPMJMS24A3 for the modeling and analysis parts.

\bibliography{references}

\clearpage
\onecolumn

\appendix

\begin{center}
    {\LARGE Supplementary Materials to \\
    ``Bayesian Inference of Discretization Error Means in ODEs via Ensemble Kalman Filtering''}
\end{center}
\vspace*{.5cm}

\section{Proof of Theorem \ref{eq:filtering_correspondance}}\label{appendix:filtering}
We prove Theorem \ref{eq:filtering_correspondance}. This result is essentially established in the classical EnKF \citep{burgers1998analysis, evensen2009data}, with the only difference being the non-zero mean in the Gaussian observation process. 

We first state two lemmas for the proof. 
\begin{lemm}[Woodbury Matrix Identity \citep{golub2013matrix}]\label{lem:woodbury}
Let $A \in \mathbb{R}^{n \times n}$ be an invertible matrix,
$U \in \mathbb{R}^{n \times k}$,
$C \in \mathbb{R}^{k \times k}$,
and $V \in \mathbb{R}^{k \times n}$.
Assume that $C$ is invertible and that
$
C^{-1} + V A^{-1} U
$
is invertible.
Then we have:
\begin{align*}
(A + U C V)^{-1}& =
A^{-1}
-
A^{-1} U
\left(
C^{-1} + V A^{-1} U
\right)^{-1}
V A^{-1}.
\end{align*}
\end{lemm}

\begin{lemm}\label{eq:Bayes_Gauus}
Let $x \sim  \mathcal{N}(\mu_1, \Sigma_1)$,
and suppose that $p(y|x)$ is given by
\begin{equation*}
    y = H x + \e
\end{equation*}
where $\e \sim \mathcal{N}(\mu_2, \Sigma_2)$ is independent of $x$. Then, $p(x|y)$ is given by 
\begin{equation*}
    \mathcal{N}\left(\mu_1 + K\left(y - H\mu_1 - \mu_2\right), \left(I - K H\right)\Sigma_1 \right),
\end{equation*}
where 
\begin{equation*}
    K:= \Sigma_1 H^{\top} (H  \Sigma_1 H^{\top} + \Sigma_2)^{-1}.
\end{equation*}
\end{lemm}

\begin{proof}
By Bayes' rule, we have 
    \begin{align*}
        p(x|y) &\propto p(y|x)  p(x)   \\
        &\propto \exp\Big\{
-\frac12
\big[
(y - Hx - \mu_2)^\top
\Sigma_2^{-1}
(y - Hx - \mu_2)
+
(x - \mu_1)^\top
\Sigma_1^{-1}
(x - \mu_1)
\big]
\Big\}.
\end{align*}

Expanding the quadratic form in $x$, we have 
\begin{align*}
&(y - Hx - \mu_2)^{\top} \Sigma_2^{-1} (y - Hx - \mu_2)
+ (x - \mu_1)^{\top} \Sigma_1^{-1} (x - \mu_1) \\
&=
\big((y-\mu_2) - Hx\big)^{\top}
\Sigma_2^{-1}
\big((y-\mu_2) - Hx\big)
+
(x - \mu_1)^{\top} \Sigma_1^{-1} (x - \mu_1) \\
&=
x^{\top} H^{\top} \Sigma_2^{-1} H x
-2 (y-\mu_2)^{\top} \Sigma_2^{-1} H x
+ (x - \mu_1)^{\top} \Sigma_1^{-1} (x - \mu_1)
+ \text{const} \\
&=
x^{\top}(\Sigma_1^{-1} + H^{\top} \Sigma_2^{-1} H)x
-2\big(
H^{\top} \Sigma_2^{-1}(y-\mu_2)
+
\Sigma_1^{-1}\mu_1
\big)^{\top} x
+ \text{const} \\
&=
\Big(
x -
(\Sigma_1^{-1} + H^{\top} \Sigma_2^{-1} H)^{-1}
\big(
H^{\top} \Sigma_2^{-1}(y-\mu_2)
+
\Sigma_1^{-1}\mu_1
\big)
\Big)^{\top} \\
&\qquad 
(\Sigma_1^{-1} + H^{\top} \Sigma_2^{-1} H) \\
&\qquad
\Big(
x -
(\Sigma_1^{-1} + H^{\top} \Sigma_2^{-1} H)^{-1}
\big(
H^{\top} \Sigma_2^{-1}(y-\mu_2)
+
\Sigma_1^{-1}\mu_1
\big)
\Big)
+ \text{const}.
\end{align*}
    Here, the term “const" represents a quantity that does not depend on $x$. 
    Therefore the posterior mean and covariance matrix are given by 
    \begin{equation}\label{eq:posteiror_means}
    (\Sigma_1^{-1}  + H^{\top} \Sigma_2^{-1}  H )^{-1}(H^{\top} \Sigma_2^{-1} (y - \mu_2) + \Sigma_1^{-1} \mu_1),       
    \end{equation}
    and    \begin{equation}\label{eq:posteiror_variances}
    (\Sigma_1^{-1}  + H^{\top} \Sigma_2^{-1}  H )^{-1}      
    \end{equation}
    respectively. 
    By applying Lemma \ref{lem:woodbury} with $A= \Sigma_1^{-1}$, $U= H^{\top}$, $C= \Sigma_2^{-1}$ and $V= H$, \eqref{eq:posteiror_variances} can be written as 
    \begin{align*}
      \eqref{eq:posteiror_variances} &= \Sigma_1 - \Sigma_1 H^{\top}(\Sigma_2 + H \Sigma_1 H^{\top})^{-1}H \Sigma_{1} \\
      & =\Sigma_1 - KH \Sigma_1 =(I - KH) \Sigma_{1}.
    \end{align*}
    Here, $K$ is defined by $K:= \Sigma_1 H^{\top}(\Sigma_2 + H \Sigma_1 H^{\top})^{-1}$. Moreover, applying Lemma \ref{lem:woodbury} once again, we obtain
\begin{align}
\eqref{eq:posteiror_means}
&=
(\Sigma_1 - \Sigma_1 H^{\top}
(H \Sigma_1 H^{\top} + \Sigma_2)^{-1}
H \Sigma_1)
\big(
H^{\top} \Sigma_2^{-1} (y - \mu_2)
+
\Sigma_1^{-1} \mu_1
\big)
\notag\\
&=
(\Sigma_1 - KH \Sigma_1)
\big(
H^{\top} \Sigma_2^{-1} (y - \mu_2)
+
\Sigma_1^{-1} \mu_1
\big)
\notag\\
&=
\mu_1
-
KH\mu_1
+
(\Sigma_1 - KH\Sigma_1)
H^{\top} \Sigma_2^{-1} (y - \mu_2).
\label{eq:posterior_mean_arranged}
\end{align}
We note that
\begin{align*}
(\Sigma_1 - KH\Sigma_1) H^{\top} \Sigma_2^{-1}
&=
\Sigma_1 H^{\top} \Sigma_2^{-1}
-
\Sigma_1 H^{\top}
(H \Sigma_1 H^{\top} + \Sigma_2)^{-1}
H \Sigma_1 H^{\top} \Sigma_2^{-1} \\
&=
\Sigma_1 H^{\top}
\Big[
\Sigma_2^{-1}
-
(H \Sigma_1 H^{\top} + \Sigma_2)^{-1}
H \Sigma_1 H^{\top} \Sigma_2^{-1}
\Big].
\end{align*}

Using the identity
\begin{align*}
\Sigma_2^{-1}
-
(H \Sigma_1 H^{\top} + \Sigma_2)^{-1}
H \Sigma_1 H^{\top} \Sigma_2^{-1}
&= (H \Sigma_1 H^{\top} + \Sigma_2)^{-1} \left( \left(H \Sigma_1 H^{\top} + \Sigma_2 \right) \Sigma_2^{-1} - H \Sigma_1 H^{\top} \Sigma_2^{-1} \right)  \\
&= (H \Sigma_1 H^{\top} + \Sigma_2)^{-1}  I = (H \Sigma_1 H^{\top} + \Sigma_2)^{-1},
\end{align*}
we obtain
\[
(\Sigma_1 - KH\Sigma_1) H^{\top} \Sigma_2^{-1}
=
\Sigma_1 H^{\top}
(H \Sigma_1 H^{\top} + \Sigma_2)^{-1}
=
K.
\]

Therefore,
\begin{align*}
\eqref{eq:posterior_mean_arranged}
&=
\mu_1
-
KH\mu_1
+
K(y-\mu_2) =
\mu_1 + K(y - H\mu_1 - \mu_2).
\end{align*}

This concludes the proof.
\end{proof}

\paragraph{Proof of Theorem \ref{eq:filtering_correspondance}} We first prove 
\begin{equation}\label{eq:convergence_firstmomment_filter}
\hat{\mu}_{i+1|i+1} \xrightarrow{P} {\mu}_{i+1|i+1}.
\end{equation}
By the definition of $\hat{\mu}_{i+1|i+1}$, we have
\begin{align}
    \hat{\mu}_{i+1 \mid i+1}  &=  \frac{1}{N_e} \sum_{n=1}^{N_e} \hat{\mu}_{i+1 \mid i+1}^n  = \frac{1}{N_e}\sum_{n=1}^{N_e}\hat{\mu}_{i+1 \mid i}^n  +\frac{1}{N_e}\sum_{n=1}^{N_e} \hat{K}_{i+1|i} \left( y_{t_{i+1}}^* + \hat{w}^n_{i+1} - H\left(x_{t_{i+1}} + \hat{\mu}_{i+1 \mid i}^n \right) \right) \notag \\
& = \hat{\mu}_{i+1 \mid i } +  \hat{K}_{i+1|i} \left( y_{t_{i+1}}^* + \hat{w}_{i+1} - H\left(x_{t_{i+1}} + \hat{\mu}_{i+1 \mid i} \right) \right), \label{eq:proof_consistency_filter}
\end{align}
where 
\begin{align*}
\hat{w}_{i+1} := \frac{1}{N_e} \sum_{n=1}^{N_e} \hat{w}_{i+1}^n, \quad
\hat{\mu}_{i+1|i} = \frac{1}{N_e} \sum_{n=1}^{N_e} \hat{\mu}_{i+1|i}^n.
\end{align*}
Since $\hat{w}_{i+1}^n \sim \mathcal{N}(0, H\Sigma H^{\top}+ \Gamma)$, we have $\hat{w}_{i+1}\xrightarrow{P} 0$.
Therefore, it follows that
\begin{align}  &\text{\eqref{eq:proof_consistency_filter}} \xrightarrow{P} {\mu}_{i+1 \mid i} +  {K}_{i+1|i} \left( y_{t_{i+1}}^*  - H\left(x_{t_{i+1}} + {\mu}_{i+1 \mid i} \right) \right), \label{eq:converge_ensemble}
\end{align}
where
\begin{equation}
    K_{i+1|i}:= {\Sigma}_{i+1 \mid i}  H^{\top} 
\left(
H  {\Sigma}_{i+1 \mid i}  H^{\top} + H\Sigma H^{\top} + \Gamma
\right)^{-1}. \label{eq:population_KalmanGain}
\end{equation} 
By Lemma \ref{eq:Bayes_Gauus}, \eqref{eq:converge_ensemble} also coincides with $\mu_{i+1|i+1}$, which proves \eqref{eq:convergence_firstmomment_filter}.

We next prove 
\begin{equation*}
\hat{\Sigma}_{i+1|i+1} \xrightarrow{P} {\Sigma}_{i+1|i+1}.
\end{equation*}
We first 
rewrite
$\hat{\mu}_{i+1 \mid i+1}$ as
\begin{align}
\hat{\mu}_{i+1 \mid i+1}^n - \hat{\mu}_{i+1 \mid i+1}  &= \hat{\mu}_{i+1 \mid i}^n +  \hat{K}_{i+1|i} \left( y_{t_{i+1}}^* + \hat{w}_{i+1}^n - H\left(x_{t_{i+1}} + \hat{\mu}_{i+1 \mid i}^n \right) \right) \notag \\
&\phantom{=}\quad- \hat{\mu}_{i+1 \mid i} -  \hat{K}_{i+1|i} \left( y_{t_{i+1}}^* + \hat{w}_{i+1} - H\left(x_{t_{i+1}} + \hat{\mu}_{i+1 \mid i} \right) \right) \notag \\
&=(I-\hat{K}_{i+1|i}H)(\hat{\mu}_{i+1 \mid i}^n- \hat{\mu}_{i+1 \mid i}) +  \hat{K}_{i+1|i} (\hat{w}_{i+1}^n- \hat{w}_{i+1}). \notag
\end{align}
Therefore, $\hat{\Sigma}_{i+1|i+1}$ can be written as 
\begin{align}
\hat{\Sigma}_{i+1 \mid i+1} &:=  \frac{1}{N_e-1} \sum_{n=1}^{N_e} \left( \hat{\mu}_{i+1 \mid i+1}^n - \hat{\mu}_{i+1 \mid i+1} \right) \left( \hat{\mu}_{i+1 \mid i+1}^n - \hat{\mu}_{i+1 \mid i+1} \right)^{\top}  \notag \\
&= \frac{1}{N_e-1} \sum_{n=1}^{N_e} \left\{ (I-\hat{K}_{i+1|i}H)(\hat{\mu}_{i+1 \mid i}^n- \hat{\mu}_{i+1 \mid i}) +  \hat{K}_{i+1|i} (\hat{w}^n_{i+1}- \hat{w}_{i+1}) \right\}  \notag \\
&~~~~~~~~~~~~~~~~~~~~~~~~~~~~~~~~~~~~ \left\{ (I-\hat{K}_{i+1|i}H)(\hat{\mu}_{i+1 \mid i}^n- \hat{\mu}_{i+1 \mid i}) +  \hat{K}_{i+1|i} (\hat{w}^n_{i+1}- \hat{w}_{i+1}) \right\}^{\top} \notag \\
&= \underbrace{ \frac{1}{N_e-1}\sum_{n=1}^{N_e} \{ (I-\hat{K}_{i+1|i}H) (\hat{\mu}_{i+1 \mid i}^n- \hat{\mu}_{i+1 \mid i})   (\hat{\mu}_{i+1 \mid i}^n- \hat{\mu}_{i+1 \mid i})^{\top}(I-\hat{K}_{i+1|i}H)^{\top} \}}_{\text{(A)}}   
\notag \\
&\phantom{=}\quad ~~~~~~~~~~~~~~~~~~~~~~~~~~~~~~~~~~~~~~~~~~~ + \underbrace{  \frac{1}{N_e-1}\sum_{n=1}^{N_e}  \left\{ \hat{K}_{i+1|i} (\hat{w}_{i+1}^n- \hat{w}_{i+1}) (\hat{w}_{i+1}^n - \hat{w}_{i+1})^{\top} \hat{K}^{\top}_{i+1|i} \right\}}_{\text{(B)}} + (\ast) 
\notag,
\end{align} 
where
\begin{align}
    (\ast)&= \frac{1}{N_e-1} \sum_{n=1}^{N_e} \left\{ (I-\hat{K}_{i+1|i} H)(\hat{\mu}_{i+1 \mid i}^n- \hat{\mu}_{i+1 \mid i})  (\hat{w}^n_{i+1}- \hat{w}_{i+1})^{\top}\hat{K}^{\top}_{i+1|i}  \right\} \notag \\
    &\phantom{=}\quad+ \frac{1}{N_e-1} \sum_{n=1}^{N_e} \left\{  \hat{K}_{i+1|i}  (\hat{w}^n_{i+1}- \hat{w}_{i+1}) (\hat{\mu}_{i+1 \mid i}^n- \hat{\mu}_{i+1 \mid i})^{\top} (I-\hat{K}_{i+1|i} H)^{\top}   \right\}. 
    \notag
\end{align}
Since $\hat{w}_{i+1}^n$ and $\hat{\mu}_{i+1 \mid i}^n$ are independent, 
$(\ast)$ converges in probability to the zero matrix $O$:
\begin{equation}\label{eq:prob_converge_first}
(\ast) \xrightarrow{P} O. 
\end{equation}
Since $\hat{\mu}_{i+1|i}^n$ and $\hat{w}^n_{i+1}$ are drawn i.i.d.~from $p(\mu_{t_{i+1}}|y_{1:i})$ and $\mathcal{N}(0, H \Sigma H^{\top} + \Gamma)$, we see that 
\begin{align}
    (A)  &\xrightarrow{P} (I-K_{i+1|i} H) \Sigma_{i+1|i}(I-K_{i+1|i} H)^{\top}, \label{eq:prob_converge_second}\\
    (B)  &\xrightarrow{P} K_{i+1|i} (H \Sigma H^{\top} + \Gamma )K^{\top}_{i+1|i}. \label{eq:prob_converge_third}
\end{align}
Equations \eqref{eq:prob_converge_first}, \eqref{eq:prob_converge_second}, \eqref{eq:prob_converge_third}
lead to 
\begin{align*}
    \hat{\Sigma}_{i+1|i+1} = \text{(A)} + \text{(B)} + (\ast)&\xrightarrow{P}(I-K_{i+1|i} H) \Sigma_{i+1|i}(I-K_{i+1|i}H)^{\top}+  K_{i+1|i} (H \Sigma H^{\top} + \Gamma)K^{\top}_{i+1|i} + O \\
    & = (I-K_{i+1|i} H) \Sigma_{i+1|i}(I-K_{i+1|i}H)^{\top}+  K_{i+1|i} (H \Sigma H^{\top} + \Gamma)K^{\top}_{i+1|i}.
\end{align*}

Finally, we see that 
\begin{align*}
    &(I-K_{i+1|i} H) \Sigma_{i+1|i}(I-K_{i+1|i}H)^{\top}+  K_{i+1|i} (H \Sigma H^{\top} + \Gamma)K^{\top}_{i+1|i}   \\
    &\quad= (I-K_{i+1|i} H)\Sigma_{i+1|i} -  \Sigma_{i+1|i}H^{\top}K^{\top}_{i+1|i}  + K_{i+1|i} H \Sigma_{i+1|i} H^{\top}K^{\top}_{i+1|i} + K_{i+1|i} (H \Sigma H^{\top} + \Gamma )K^{\top}_{i+1|i}  \\
    & \quad=( I-K_{i+1|i}H) \Sigma_{i+1|i} -  \Sigma_{i+1|i}H^{\top}K^{\top}_{i+1|i}  + K_{i+1|i}(H \Sigma_{i+1|i}H^{\top} +  H \Sigma H^{\top} + \Gamma )K^{\top}_{i+1|i} \\
    & \quad=( I-K_{i+1|i}H) \Sigma_{i+1|i} -  \Sigma_{i+1|i}H^{\top}K^{\top}_{i+1|i}  \\
    &\quad \phantom{=}\quad + {\Sigma}_{i+1 \mid i}  H^{\top} 
\left(
H  {\Sigma}_{i+1 \mid i}  H^{\top} + H \Sigma H^{\top} + \Gamma
\right)^{-1} (H \Sigma_{i+1|i}H^{\top} +  H \Sigma H^{\top} + \Gamma )K^{\top}_{i+1|i} \\
&\quad =( I-K_{i+1|i}H) \Sigma_{i+1|i} -  \Sigma_{i+1|i}H^{\top}K^{\top}_{i+1|i}  + {\Sigma}_{i+1 \mid i}  H^{\top} K^{\top}_{i+1|i} \\
    &\quad =( I-K_{i+1|i}H) \Sigma_{i+1|i}.
\end{align*}
Here, the third equation is derived from  the definition \eqref{eq:population_KalmanGain} of $K_{i+1|i}$.
This concludes the proof, with the use of Lemma \ref{eq:Bayes_Gauus}.

\section{Derivation of Eq.~\eqref{eq:simplified_smoothing}.}\label{appendix:smoothing}
We first note that, under the augmented state-space model, the observation operator $\boldsymbol{\tilde{H}}$ is given by
\begin{equation}
    \boldsymbol{\tilde{H}} = (H, O, \ldots, O),
\end{equation}
where $O$ denotes a $d_{\Y} \times d_{\X}$ zero matrix. Assume that we have an ensemble $\{ \hat{\boldsymbol{\mu}}_{i+1|i}^n \}_{n=1}^{N_e}
 = \left\{\left( \hat{\mu}_{i+1|i }^n, \ldots, \hat{\mu}_{i+1-L|i}^n  \right)^{\top} \right\}_{n=1}^{N_e}.
$  of $p(\boldsymbol{\mu}_{i+1} | y_{1:i})$. Let $\hat{\boldsymbol{\Sigma}}_{i+1|i}$ be a sample covariance matrix of the ensemble. Then, the 
correction step for the augmented state-space model updates the ensemble $\{ \hat{\boldsymbol{\mu}}_{i+1|i}^n \}_{n=1}^{N_e}$ by 
\begin{align} &\hat{\boldsymbol{\mu}}_{i+1 \mid i+1}^n := \notag \\
&~~~~~~~\hat{\boldsymbol{\mu}}_{i+1 \mid i}^n + \hat{\boldsymbol{\Sigma}}_{i+1 \mid i}  \boldsymbol{\tilde{H}}^{\top} 
\left(
\boldsymbol{\tilde{H}} \hat{\boldsymbol{\Sigma}}_{i+1 \mid i}  \boldsymbol{\tilde{H}}^{\top} + H \Sigma H^{\top} + \Gamma
\right)^{-1}  \left( y_{t_{i+1}}^* + \hat{w}^n_{i+1} - Hx_{t_{i+1}} - \boldsymbol{\tilde{H}}\hat{\boldsymbol{\mu}}_{i+1 \mid i}^n  \right).  \label{eq:smoothing_ensemble}
\end{align}
It follows from 
the definition of $\boldsymbol{\tilde{H}}$
that
\begin{align}
   \boldsymbol{\tilde{H}}\hat{\boldsymbol{\mu}}_{i+1 \mid i}^n &= (H, O, \ldots, O)   (\hat{\mu}_{i+1 \mid i}^n, \hat{\mu}_{i \mid i}^n,\ldots, \hat{\mu}_{i-L+1 \mid i}^n)^{\top} = H\hat{\mu}_{i+1 \mid i}^n. \label{eq:smoothing_supporting1}
\end{align}
Noting that
\begin{align*}
\hat{\boldsymbol{\Sigma}}_{i+1 \mid i}
=
\begin{pmatrix}
\hat{\Sigma}_{i+1,i+1|i} & \hat{\Sigma}_{i+1,i|i} & \cdots & \hat{\Sigma}_{i+1,i-L+1|i} \\
\hat{\Sigma}_{i,i+1|i} & \hat{\Sigma}_{i,i|i} & \cdots & \hat{\Sigma}_{i,i-L+1|i} \\
\vdots & \vdots & \ddots & \vdots \\
\hat{\Sigma}_{i-L+1,i+1|i} & \hat{\Sigma}_{i-L+1,i|i} & \cdots & \hat{\Sigma}_{i-L+1,i-L+1|i}
\end{pmatrix},
\end{align*}
where $\hat{\Sigma}_{j, k \mid i} := \frac{1}{N_e-1} \sum_{n=1}^{N_e} \left( \hat{\mu}_{j \mid i}^n - \hat{\mu}_{j \mid i} \right) \left( \hat{\mu}_{k \mid i}^n - \hat{\mu}_{k \mid i} \right)^{\top}$ denotes the sample cross-covariance between $\{ \hat{\mu}_{j \mid i}^n \}_{n=1}^{N_e}$ and
$\{ \hat{\mu}_{k \mid i}^n \}_{n=1}^{N_e}$,
we have  
\begin{align}
\hat{\boldsymbol{\Sigma}}_{i+1 \mid i}  \boldsymbol{\tilde{H}}^{\top} &= (\hat{\Sigma}_{i+1,i+1|i},\hat{\Sigma}_{i,i+1|i}, ..., \hat{\Sigma}_{i-L+1,i+1|i} )^{\top} H^{\top}, \label{eq:smoothing_supporting2}\\
   \boldsymbol{\tilde{H}}  \hat{\boldsymbol{\Sigma}}_{i+1 \mid i}  \boldsymbol{\tilde{H}}^{\top} &= H  \hat{\Sigma}_{i+1, i+1|i}  H^{\top}. \label{eq:smoothing_supporting3}
\end{align}
Applying \eqref{eq:smoothing_supporting1}, \eqref{eq:smoothing_supporting2} and \eqref{eq:smoothing_supporting3} to \eqref{eq:smoothing_ensemble} concludes the proof.

\section{Proof of Theorem \ref{theo_convergence_rate_revised}}\label{appendix_proof_prior}

In this appendix, we prove Theorem~\ref{theo_convergence_rate_revised}. 
The proof proceeds in essentially the same manner as that of Theorem~4.1 in \cite{toyota2025joint}.

While \eqref{eq:prior_stepsize} defines the Markov prior only at the observation points 
$t_0, \ldots, t_N$, the recursive construction in \eqref{eq:proposed_prior} naturally allows us 
to extend this definition to finer time grids as follows:
\begin{align*}
    p_h (\mu_{t_{0, 0}},\mu_{t_{0, 1}},...,\mu_{t_{0, k-1}},\mu_{t_{1, 0}}   \ldots, \mu_{t_{N, 0}}) 
    = p_h (\mu_{t_{0, 0}}) \prod_{i=0}^{N-1}  \prod_{j=0}^{k-1} 
    p_h (\mu_{t_{i, j+1}} \mid \mu_{t_{i, j}}). 
\end{align*}

\begin{proof}[Proof of Theorem \ref{theo_convergence_rate_revised}]
Let us recall that random variables $\{ X_h \}_{h>0}$ and $\{ Y_h \}_{h>0}$ indexed by $h$ satisfy 
$X_h  = \mathcal{O}_p (Y_h)$ if $X_h/Y_h$ is uniformly tight:
\[
\lim_{\epsilon \rightarrow \infty} 
\sup_{h}\mathbb{P}\left( \left|\frac{X_h}{Y_h} \right| \geq \epsilon \right)  = 0.
\]
By Markov inequality, we have
\begin{equation}
    \sup_{h} \mathbb{P}\left(\left\lVert \frac{\mu_{t_i} }{h^{\min({a}, {b})}} \right\rVert \geq \epsilon \right) 
    \leq \sup_{h} \frac{1}{\epsilon h^{\min({a}, {b})}} 
    \mathbb{E}_{\mu_{t_i} \sim p_h (\mu_{t_i})}[\| \mu_{t_i}  \|].
\end{equation}
Hence, it suffices to prove that 
\[
\mathbb{E}_{p_h (\mu_{t_i})}[\| \mu_{t_i}  \|] 
= \mathcal{O}(h^{\min({a}, {b})})
\]
for $i \geq 1$.

Define 
\[
e_{i, j} := \mathbb{E}_{p_h(\mu_{t_{i, j}})}[\| \mu_{t_{i, j}} \|]
\]
to simplify the notation. 
By the definition \eqref{eq:proposed_prior}, it holds that for 
$0 \leq i \leq (N -1)$ and $0 \leq j \leq k-1$, 
\begin{equation*}
\mu_{t_{i, j+1}}  = M_{i, j}  \mu_{t_{i, j}}  + L\left(t_{i, j}  \right).
\end{equation*}
This leads to the inequality 
\begin{align}
    e_{i, j+1} 
    &= \mathbb{E}_{p_h(\mu_{t_{i, j+1}})}[\| \mu_{t_{i, j+1}} \|]  \notag \\
    &= \mathbb{E}_{p_h(\mu_{t_{i, j}}), P_{\lambda}}[\| M_{i, j}  \mu_{t_{i, j}} + L(t_{i, j}) \|]\notag \\[1ex]
    &\leq \mathbb{E}_{p_h(\mu_{t_{i, j}}), P_{\lambda}}[\| M_{i, j}  \mu_{t_{i, j}}  \|] + \| L(t_{i, j}) \| \notag\\[1ex]
    &\leq \mathbb{E}_{p_h(\mu_{t_{i, j}}), P_{\lambda}}[\|M_{i, j} \|_{\text{op}}  \|\mu_{t_{i, j}}  \|] 
    +  \| L(t_{i, j}) \| \notag \\[1ex]
    &= \mathbb{E}_{P_{\lambda}}[\|M_{i, j} \|_{\text{op}} ] 
     \mathbb{E}_{p_h(\mu_{t_{i, j}})}[\|\mu_{t_{i, j}}  \|] 
    +  \| L(t_{i, j}) \|  \label{eq:convergence rate_indep} \\[1ex]
    &=  \mathbb{E}_{P_{\lambda}} [\| I - ( I  - M_{i, j} )\|_{\text{op}}] 
     e_{i, j}  +  \| L(t_{i, j}) \| \notag  \\
    &\leq \left\{ \mathbb{E}_{P_{\lambda}} [\| I \|_{\text{op}}] 
    + \mathbb{E}_{P_{\lambda}} [\|  I  - M_{i, j} \|_{\text{op}}] \right\} 
     e_{i, j} +  \| L(t_{i, j}) \| \notag   \\
    &\leq (1 +C_M h)  e_{i, j} + C_{L}  h^{{a} +1}. 
    \label{eq:convergence rate2}
\end{align}
Here, $\mathbb{E}_{p_h(\mu_{t_{i,j}}),\,P_{\lambda}}$ denotes the expectation with respect to 
$\mu_{t_{i,j}} \sim p_h(\mu_{t_{i,j}})$ and $M_{i,j} \sim P_{\lambda}$. The equality \eqref{eq:convergence rate_indep} is derived by Condition (IV). The inequality \eqref{eq:convergence rate2} is derived from Conditions 
(I) and (II). 
Applying \eqref{eq:convergence rate2} recursively yields
\begin{align}
    e_{i, j+1} 
    &\leq (1+C_M h)^{j+1+ik}  e_{0, 0} 
    + C_L h^{a +1} 
      \sum_{l=1}^{j+ik+1} (1+C_M h)^{l-1}\notag  \\
    &= (1+C_M h)^{j+1+ik}  e_{0, 0} 
    + C_L h^{a } 
      \frac{(1+C_M h)^{j+1+ik} - 1}{L}  \notag \\
    &\leq \exp\big( C_M  (t_{i, j+1} - t_{0}) \big)   e_{0, 0} 
    + C_L h^{a} 
      \frac{\exp\big( C_M (t_{i, j+1} - t_{0}) \big) - 1}{C_M }.
\end{align}

Since this holds for $j = k -1$ and $0 \leq i \leq N-1$, we obtain
\begin{align*}
  \mathbb{E}_{p_h (\mu_{t_{i}})}[\| \mu_{t_i}  \|] 
  &= e_{i-1, k}  \notag \\
  &\leq \exp\big( C_M ( t_{i} - t_0 ) \big) e_{0, 0} 
  + C_L h^{a} 
    \frac{\exp\big( C_M ( t_{i} - t_0 ) \big) - 1}{C_M } \\
    &\leq \exp\big( C_M ( t_{i} - t_0 ) \big) C_{\mu_{}} h^b
  + C_L h^{a} 
    \frac{\exp\big( C_M ( t_{i} - t_0 ) \big) - 1}{C_M } \\
    &\leq \left\{ \exp\big( C_M ( t_{i} - t_0 ) \big) C_{\mu_{}} 
  + C_L
    \frac{\exp\big( C_M ( t_{i} - t_0 ) \big) - 1}{C_M }  \right\}  h^{\min (a, b)}
\end{align*}
for $h \leq 1$, 
which concludes the proof. Here, the final inequality is derived from Condition (III).
\end{proof}

\section{Related Work}

In this section, we briefly review uncertainty quantification methods for differential equations and related techniques based on the Ensemble Kalman Filter.

\paragraph{Probabilistic Numerics} For the quantification of discretization error in differential equations, a recently well-established paradigm is known as probabilistic numerics (PN)\citep{cockaynebayeisna2019,hennig2022probablistic}. By viewing numerical computation as a statistical inference problem, PN provides numerical solutions together with principled uncertainty quantification of discretization errors. As discussed in Section \ref{sec:general_layout}, Bayesian ODE solvers ~\citep{beck2024diffusiontemper, kersting2020, le2025modelling, schmidt2021sir, SchoberSarkkaHennig2018, pmlr-v162-tronarp22a, Tronarp2018, TroSarHen2021, dingling2025propagating, kramer2025adaptive} are among the most developed approaches within PN for ordinary differential equations. These methods place a prior distribution on the solution and its derivatives, and treat the constraint that the solution satisfies the differential equation as an observation (referred to as an information operator in the PN literature~\citep{cockaynebayeisna2019}). Conditioning on these observations yields a posterior distribution over the numerical solution.
An alternative approach is the perturbation-based method~\citep{abdulle2020random, conrad2017statistical, lie2022randomised, lie2019strong}. Rather than placing a prior directly on the solution, these methods intentionally introduce stochastic perturbations into the numerical solver and quantify uncertainty through the resulting distribution of numerical trajectories.

These approaches differ from the present work in two important respects. First, Bayesian ODE solvers place a Markov prior directly on the solution of the differential equation, whereas in this study the Markov prior is imposed on the discretization-error trajectory. By placing the prior on the discretization error itself, we can incorporate insights from classical error analysis into the prior model, which is difficult when the prior is defined on the solution. Second, unlike other PN approaches, our method infers the discretization error directly from observations. This perspective enables the simultaneous updating of the discretization error, alongside the standard targets of ODE inverse problems (e.g., model parameters), using observational data.

\paragraph{Extensions of Ensemble Kalman Filters}

Although our implementation of the EnKF assumes a linear observation process, it is important to note that a variety of extensions have been developed to accommodate nonlinear observation and state-transition models. A widely used approach is the Extended Kalman Filter (EKF) \citep{sarkka2023bayesian}, which locally linearizes nonlinear mappings via a first-order Taylor expansion based on the Jacobian matrix. However, this procedure introduces approximation errors inherent to the linearization. 
To this end, alternative methods have been developed that use only the outputs of the observation process generated from ensemble inputs, without requiring an explicit analytical representation of the observation operator~\citep{katzfuss2016understanding, evensen2009data}.

In our experiments, we employ fixed-lag smoothing. Since the primary objective of the present study is to demonstrate that estimating the discretization error mean can be reduced to a problem setting in which the EnKF can be directly applied without approximation, we adopt the classical fixed-lag smoothing approach.
We would like to emphasize, however, that other smoothing strategies may also improve performance. In particular, fixed-interval smoothing \citep{sarkka2023bayesian}, which targets the full time series, is a natural alternative for implementing our framework. By exploiting all available observations, it enables inference that more fully utilizes observational information.
Recent results have also shown that fixed-point smoothing can be implemented with low computational cost while retaining numerical stability and efficiency \citep{kramer2025adaptive}. These techniques therefore provide another computationally efficient option.

\section{Additional Experiment on the Lorenz--96 Model}\label{appendix_Lorenz96}

To evaluate the effectiveness of the proposed method in a high-dimensional setting, we consider the Lorenz--96 model, a standard benchmark system for data assimilation and chaotic dynamics. The system is defined by the following set of coupled ordinary differential equations:
\begin{equation}
\frac{d x_i(t)}{dt}
= \bigl(x_{i+1}(t) - x_{i-2}(t)\bigr)x_{i-1}(t)
- x_i(t) + F,
\quad i = 1,\ldots, d_\X, 
\end{equation}
with periodic indexing. In this experiment, we set the system dimension to $d_\X = 8$ and the forcing term to $F = 8.0$. The system is numerically solved using the Euler method with a step size of $0.01$. The observation operator $H$ is defined as the identity matrix, and observations are assumed to be available at $t = 1.0, 2.0, \ldots, 10.0$, with a fixed covariance matrix $\Gamma = \mathrm{diag}(1.0^2, 1.0^2, \ldots, 1.0^2)$.

The initial prior distribution is given by a multivariate Gaussian:
$
p(\mu_{t_0}) = \mathcal{N}(0, I).
$
The prior distribution is constructed following the same procedure as in Section~4. 
The hyperparameters $(\alpha, \beta, \gamma)$ are selected by maximizing the marginal likelihood over the candidate sets
$
\alpha \in \{-1.4, -1.2, \dots, 1.4\}, \quad
\beta \in \{0.1, 0.2, \dots, 1.0\}, \quad
\gamma \in \{0.5, 1.0, \dots, 4.0\}.
$
We use an ensemble size of $100$ for the Ensemble Kalman Filter (EnKF). For comparison, we implement a particle filter (PF) with $100{,}000$ particles. We attempted to select the hyperparameters using the same candidate sets and marginal-likelihood criterion as in the EnKF experiments. 
However, 
the computed log-marginal likelihood overflowed to infinity for all candidate parameter combinations.
A possible explanation is severe particle degeneracy, whereby the particles failed to adequately represent the true state and the total particle weight became numerically close to zero. 
Accordingly, we fix the hyperparameters at $(\alpha, \beta, \gamma) = (1.0, 1.0, 1.0)$.

\begin{table}[h]
\centering
\begin{tabular}{c|cccccccccc}
\hline
Method & $t=1$ & $t=2$ & $t=3$ & $t=4$ & $t=5$ & $t=6$ & $t=7$ & $t=8$ & $t=9$ & $t=10$ \\
\hline
EnKF & 0.7329 & 0.5500 & 0.6625 & 0.4909 & 0.7585 & 0.7389 & 0.6655 & 0.5939 & 1.0142 & 1.3573 \\
PF   & 0.1320 & 0.3055 & 0.9591 & 2.5171 & 2.2915 & 2.2845 & 3.5753 & 3.1893 & 1.2671 & 2.6477 \\
\hline
\end{tabular}
\caption{Estimation error comparison between EnKF and PF over time steps.}\label{tab:enkf_pf_lorenz96}
\end{table}

Table~\ref{tab:enkf_pf_lorenz96} reports the performance of EnKF and PF at each observation time. Their performance is evaluated using the absolute difference between the true discretization error and the ensemble mean estimate in EnKF, and between the true discretization error and the particle mean estimate in PF.
The results show that the performance of PF deteriorates rapidly after  $t\approx3$. This degradation is likely due to particle degeneracy, which becomes increasingly severe over time and prevents the PF from accurately representing the underlying posterior distribution. In contrast, EnKF maintains relatively stable performance and consistently provides more accurate estimates of the discretization error at later observation times. These results suggest that the proposed EnKF-based approach is more robust than PF for high-dimensional dynamical systems such as the Lorenz--96 model.

Figure~\ref{fig:disc_error_8d} compares the true discretization error with the estimated discretization error means obtained using EnKF and PF. The figure further demonstrates that the EnKF estimate closely follows the trajectory of the true discretization error, whereas PF fails to accurately capture the evolution of the discretization error in the eight-dimensional setting, even when $100{,}000$ particles are used.

\begin{figure}[t]
\vspace{5mm}
    \centering
    \begin{subfigure}[b]{\textwidth}
        \centering
        \includegraphics[width=\textwidth]{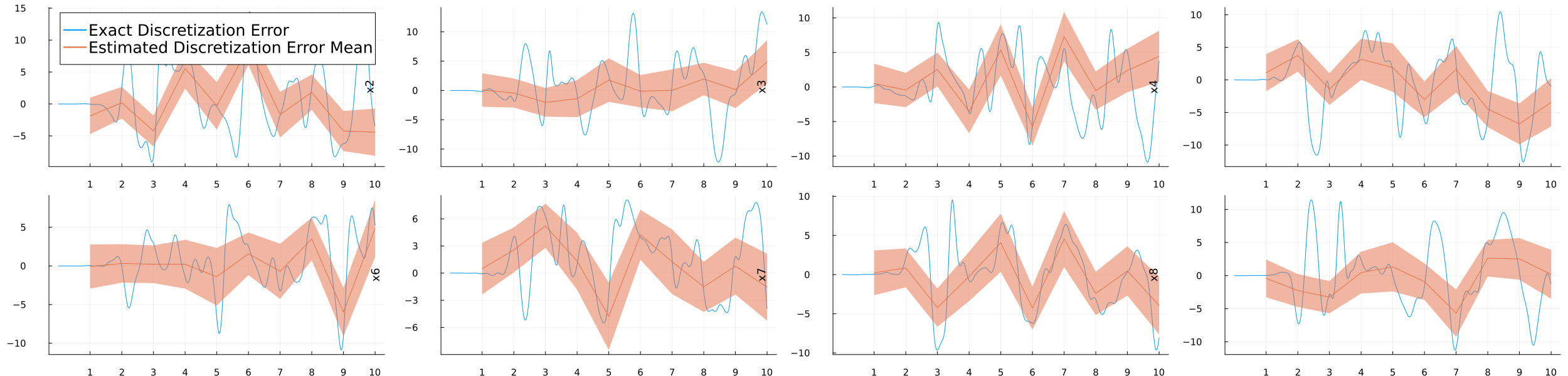}
        \caption{EnKF-based posterior on discretization error mean and true discretization error}
        \label{fig:disc_error_enkf_8d}
    \end{subfigure}


    \begin{subfigure}[b]{\textwidth}
        \centering
        \includegraphics[width=\textwidth]{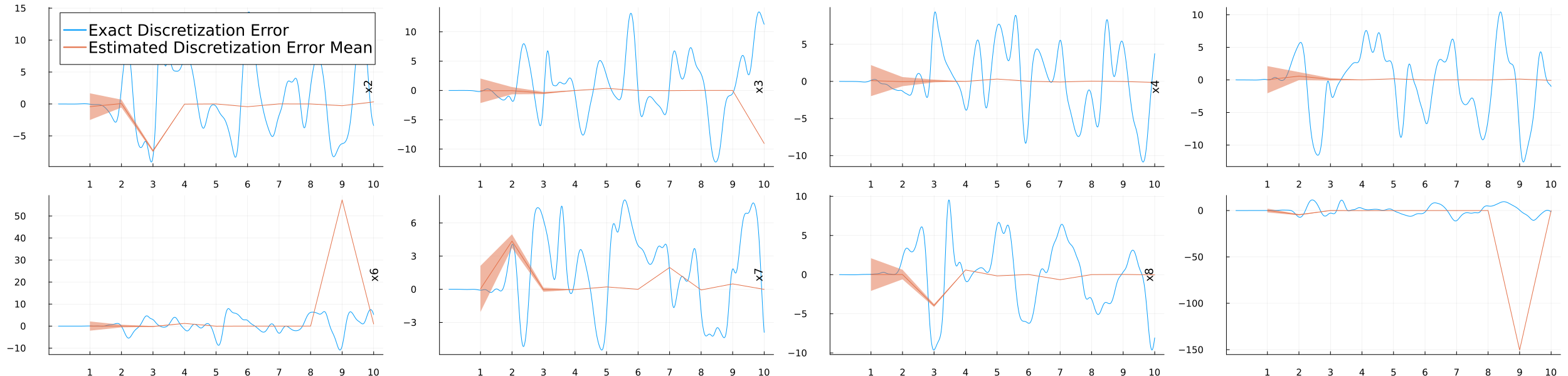}
        \caption{PF-based posterior on discretization error mean and true discretization error}
        \label{fig:disc_error_pf_8d}
    \end{subfigure}

    \caption{
Discretization error quantification results for the $8$-dimensional Lorenz--96 model. The posterior mean and $95\%$ credible interval for the mean discretization error are reported and compared with the true discretization error.
    }
    \label{fig:disc_error_8d}
\end{figure}

\end{document}